\documentclass[12pt,leqno]{amsart}

 \usepackage{times}
\usepackage{amsmath,amsfonts,amstext,amssymb,amsbsy,amsopn,amsthm, mathrsfs}
\usepackage[latin1]{inputenc}
\usepackage[
maxbibnames=99,
backend=biber,
style=numeric,
sorting=nyt
]{biblatex}

\usepackage{hyperref}
\usepackage{accents}
\usepackage{enumerate}

\newtheorem{theorem}{Theorem}[section]
\newtheorem{proposition}[theorem]{Proposition}
\newtheorem{lemma}[theorem]{Lemma}
\newtheorem{corollary}[theorem]{Corollary}

\theoremstyle{definition}
\newtheorem{definition}[theorem]{Definition}

\theoremstyle{remark}
\newtheorem{remark}{Remark}[section]

\newcommand{\cal}[1]{\mathcal{#1}}
\newcommand{\bb}[1]{\mathbb{#1}}

\renewcommand{\t}[1]{\text{#1}}

\renewcommand{\bar}[1]{\overline{#1}}
\renewcommand{\tilde}{\widetilde}

\newcommand{\pt}{\partial}
\DeclareMathOperator{\tr}{tr}
\DeclareMathOperator{\sinc}{sinc}
\DeclareMathOperator{\supp}{supp}
\DeclareMathOperator{\Ric}{Ric}
\DeclareMathOperator{\vol}{\! vol}
\DeclareMathOperator{\Sec}{Sec}
\DeclareMathOperator{\Hess}{Hess}

\DeclareMathOperator{\diam}{diam}

\DeclareMathOperator{\ac}{{ac}}

\DeclareMathOperator{\sff}{\text{II}}

\mathchardef\mhyphen="2D

\numberwithin{equation}{section}
\allowdisplaybreaks

\begin{document}

\title[OT Approach to Michael-Simon-Sobolev Inequalities]
{Optimal Transport Approach
to Michael-Simon-Sobolev Inequalities 
in Manifolds
with Intermediate Ricci Curvature Lower Bounds}

\author{Kai-Hsiang Wang}
\address{Department of Mathematics, Northwestern University, 2033 Sheridan Road, Evanston, Illinois IL 60208, USA}
\email{khwang2025@u.northwestern.edu}

\thanks{Affiliation: Department of Mathematics, Northwestern University, 2033 Sheridan Road, Evanston, Illinois IL 60208, USA. 
Email: khwang2025@u.northwestern.edu}


\subjclass[2020]{Primary 53C21, 49Q22; Secondary 53C23}

\keywords{optimal transport, McCann's theorem, Michael-Simon inequality, intermediate Ricci curvature}

\begin{abstract}
    We generalize McCann's theorem of optimal transport to a submanifold setting and use it to prove Michael-Simon-Sobolev inequalities for submanifolds in manifolds with lower bounds on intermediate Ricci curvatures.
    The results include a variant of the sharp Michael-Simon-Sobolev inequality in S. Brendle's [arXiv:2009.13717] when the intermediate Ricci curvatures are nonnegative.
\end{abstract}

\maketitle
Note: This preprint has not undergone peer review (when
applicable) or any post-submission improvements or corrections. The Version of Record of this
article is published in Annals of Global Analysis and Geometry, and is available online at \url{https://doi.org/10.1007/s10455-023-09934-9}
\section{Introduction}
    
    This article is two-fold.
    In the first part, we generalize McCann's theorem \cite{McC01} in optimal transport theory to a submanifold setting and obtain a corresponding change of variable formula.
    In the second part, we use these results to give  optimal transport proofs of Michael-Simon-Sobolev inequalities for compact submanifolds in complete Riemannian manifolds with intermediate Ricci curvatures  bounded from below.
    Most of the arguments are inspired by S. Brendle's \cite{Bre22}, where the Alexandrov-Bakelman-Pucci (ABP) maximum principle was used; a hint to use optimal transport can be found in Brendle's earlier work \cite{Bre21}.
    
    \subsection{Generalization of McCann's Theorem}
    In optimal transport theory, McCann's theorem provides the existence and description of optimal transport maps from one Borel probability measure to another on a complete Riemannian manifold when the cost is quadratic (distance squared) and the source measure is absolutely continuous w.r.t.\@ the volume measure.
    We refer to Section \ref{Sec_prelim} for the terminology and classical results in optimal transport theory.
    \par
    In the first part of this article, we study an optimal transport problem  from a lower dimension to the top dimension on a complete Riemannian manifold with the quadratic cost, in the sense that the two measures are absolutely continuous with respect to the volume measure of a submanifold and that of the ambient manifold, respectively.
    Our first result is the following:
    \begin{theorem}[Lemma \ref{Tangency_submfd}, Theorem \ref{gen_McCann_thm}\label{Intro_gen_McCann_thm}, Corollary \ref{superdiff_normalfib}]
    Let $M$ be a complete Riemannian manifold and $d_M$ the geodesic distance function.
    Let $\Sigma\subset M$ be a compact submanifold, possibly with boundary, and $\Omega\subset M$ a compact regular domain.
    Let $\mu$ and $\nu$ be Borel probability measures on $\Sigma$ and $\Omega$, respectively, that are absolutely continuous with respect to the corresponding volume measures.
    Let $\phi$ be a Kantorovich potential function on $\Sigma$ associated to an optimal transport plan of $\mu$ and $\nu$ with the cost function $c:=\frac{1}{2}d_M^2$ on $\Sigma \times \Omega$. 
    Then there is a measurable subset $A$ of the normal bundle $T^\perp \Sigma$ and a map $\Phi:A\to \Omega$ defined by
    \begin{equation}
        \Phi(x,v)=\exp^M_{x}(-\nabla^\Sigma \phi(x)+v)
    \end{equation}
    where $x\in \Sigma$ and $v\in T^\perp_x \Sigma$, such that 
    \begin{enumerate}
        \item $p_{T^\perp \Sigma}(A)$ has full $\mu$-measure in $\Sigma$, where $p_{T^\perp \Sigma}:T^\perp \Sigma \to \Sigma$ is the bundle projection;
        
        \item $\Phi(A)$ has full $\nu$-measure in $\Omega$;
        
        \item For $(x,v)\in A$, the curve $\gamma(t):=\exp^M_x(-t\nabla^\Sigma \phi(x)+tv)$ for $t\in [0,1]$ is a minimizing geodesic;
   
    \item The restricted map at each $x\in p_{T^\perp \Sigma}(A)$ gives a bijection:
    \begin{equation}
        \Phi(x,\cdot):A\cap T^\perp_x \Sigma \to \pt^{c_+}\phi(x)\cap \Phi(A).
    \end{equation}
    \end{enumerate}
    \end{theorem}
    Combining Theorem \ref{Intro_gen_McCann_thm} above with the transport condition, we deduce the following change of variable formula:
    \begin{theorem}[Theorem \ref{change_of_var}]\label{Intro_change_of_var}
    With the same notations as in Theorem \ref{Intro_gen_McCann_thm}, suppose $f$ and $g$ are the density functions of $\mu$ and $\nu$ w.r.t.\@ the volume measures of $\Sigma$ and $M$, respectively. 
    Then for $\mu$-a.e.\@ $x\in \Sigma$, we have
    \begin{equation*}
    \int_{A\cap T^\perp_x \Sigma}g\circ \Phi(x,v)|\det D\Phi(x,v)| dv= f(x).
    \end{equation*}
    \end{theorem}
    
    Roughly speaking, Theorem \ref{Intro_gen_McCann_thm} identifies the $c$-superdifferential $\pt^{c_+}\phi(x)$ with a specific part of the normal fiber up to null subsets.
    Compared with the classical McCann's theorem, where  point-wise transportation takes place, our result matches the intuition that in order to compensate the dimensional difference between the source and the target measure, the set-valued map $\pt^{c_+}\phi$ spreads out the point mass at $x$ over the fiber $A\cap T^\perp_x \Sigma$ of dimension equal to the codimension of the submanifold, which is exactly the content of Theorem \ref{Intro_change_of_var}.
    On the other hand, we also observe the same phenomenon, a.k.a.\@ "displacement interpolation," that mass transportation occurs along minimizing geodesics.
    \par    
    In the literature, optimal transport problems involving submanifolds or dimensional differences have been considered by several authors, e.g.,
\begin{itemize}
    \item In W. Gangbo and R. J. McCann's \cite{GM00}, they considered an optimal transport problem between measures on hypersurfaces in the Euclidean space.
    Their result about the specific shape of $\pt^{c_+}\phi$ (Corollary 1.5, Lemma 1.6 in \cite{GM00}) can almost be recovered by Theorem \ref{Intro_gen_McCann_thm} (more precisely by Lemma \ref{Tangency_submfd}) since the map $\Phi$ becomes linear in the variable $v$ of normal vectors when the ambient space $M$ is Euclidean.
    
    \item P. Castillon's \cite{Cas10} proved a sharp weighted Michael-Simon-Sobolev inequality for submanifolds in the Euclidean space.
    In the article, Castillon studied an optimal transport problem from a measure on a submanifold to a measure on a linear subspace of the same dimension as the submanifold, and the arguments involved Euclidean orthogonal projections.
    In contrast, our approach is closer to the classical optimal transport problem by "raising" the dimension to the top and considering the transportation between objects of the same top dimension (the set $A$ and $\Omega$ in Theorem \ref{Intro_gen_McCann_thm}).
    
    \item B. Pass's \cite{Pas12} and McCann and Pass's \cite{MP20} studied an optimal transport problem from a higher dimension to a lower dimension with general cost functions.
    This setting is actually considered in the proof of Theorem \ref{Intro_gen_McCann_thm} (see Theorem \ref{gen_McCann_thm}), and the foliation structure in \cite{Pas12} is realized as the normal bundle structure in our case.
    In addition, Theorem 2 in \cite{MP20} is a similar change of variable formula as Theorem \ref{Intro_change_of_var}.
    
    \item C. Ketterer and A. Mondino's \cite{KM18} considered an optimal transport problems between measures of the same dimension which is smaller than that of the ambient space.
    More precisely, these measures are supported on rectifiable subsets of the same smaller dimension and absolutely continuous with respect to the corresponding Hausdorff measure.
    We are inspired by \cite{KM18} to use lower bounds on the intermediate Ricci curvature in the proof of Michael-Simon-Sobolev inequalities in Section \ref{Sec_Michael_Simon}.
    
    \item A recent paper \cite{BE22} by S. Brendle and M. Eichmair used the Kantorovich dual functions from optimal transport theory to prove a  Michael-Simon-Sobolev inequality in the Euclidean space.
    In comparison, we use the graph-like property of the support of an optimal transport plan (see e.g. Theorem 2.13 in L. Ambrosio and N. Gigli's \cite{AG13}; also see Theorem \ref{funda_OT}).
\end{itemize}

    To the best of our knowledge, the optimal transport problem we consider here from a lower dimension to the top dimension is new.
    We speculate future application of these results and possible generalizations to non-smooth settings, e.g. for subspaces in metric measure spaces.

    \subsection{Michael-Simon-Sobolev inequalities}
    In the second part of this article, we use the results from the first part to prove Michael-Simon-Sobolev inequalities for compact submanifolds in complete Riemannian manifolds with intermediate Ricci curvatures bounded from below.
    The results will be divided into 3 cases where the lower bounds are either zero, positive, or negative.
        \par
    In the literature, intermediate Ricci curvature lower bounds have been used to prove geometric inequalities for submanifolds, e.g.\@ Brunn-Minkowski type inequalities in \cite{KM18} and a generalization of Heintze-Karcher Comparison \cite{HK78} in Y. K. Chahine's \cite{Cha20} by replacing the sectional curvature condition with integral bounds on the intermediate Ricci curvatures.

    \subsubsection{Nonnegative Intermediate Ricci curvatures}
    We begin with the cleanest case where the intermediate Ricci curvatures are nonnegative.
    The following partially sharp inequality is a variant of Theorem 1.4 in \cite{Bre22}:
    \begin{theorem}\label{Intro_Mic_Sim_interRicci_nonneg}
        Let $M$ be a complete noncompact Riemannian manifold of dimension $n+m$ with $\Ric_n, \Ric_m\geq 0$ and $m\geq 2$.
        Let $\theta$ be the asymptotic volume ratio of $M$:
        \begin{equation}\label{def_AVR}
            \theta:=\lim_{r\to \infty} \frac{\vol_M(B_r(x_0))}{|B^{n+m}|r^{n+m}}
        \end{equation}
        for some (any) point $x_0\in M$.
        Let $\Sigma\subset M$ be a compact submanifold of dimension $n$, possibly with boundary $\pt \Sigma$. Let $H$ be the mean curvature vector of $\Sigma$.
        Let $f$ be a positive smooth function on $\Sigma$.
        Then
        \begin{equation*}
        n\theta^\frac{1}{n}\left(\frac{(n+m)|B^{n+m}|}{m|B^m|}\right)^\frac{1}{n}\left( \int_\Sigma f^\frac{n}{n-1}\right)^\frac{n-1}{n}\leq\int_\Sigma |\nabla^\Sigma f |+\int_{\pt \Sigma}f+\int_\Sigma f|H|.
        \end{equation*}
    Here, $\Ric_n$ and $\Ric_m$ are the intermediate $n$- and $m$-Ricci curvatures of $M$ (see the definitions in Section \ref{Sec_Michael_Simon}), and $|B^k|$ is the volume of the unit ball in $\bb{R}^k$.
    \end{theorem} 
    
    Compared with Theorem 1.4 in \cite{Bre22}, we here use a weaker assumption that the intermediate Ricci curvatures of the ambient manifold are nonnegative, although our conclusion is also weaker.
    Nevertheless, our inequality is still sharp when $m=2$, and in this case the curvature condition is equivalent to $\Sec\geq 0$ (See Remark \ref{Rmk_interRicci}).
    Additionally, as Theorem 1.4 in \cite{Bre22}, our inequality also holds (sharply) for hypersurfaces with $m=1$ since we can raise the codimension by considering $M\times \bb{R}$ while preserving the asymptotic volume ratio and the curvature condition $\Ric_1\geq 0$, which is again equivalent to $\Sec\geq 0$.
    We remark that the intermediate Ricci curvature assumptions also apply to the original argument in \cite{Bre22}.
    \par
    Our optimal transport proof of Theorem \ref{Intro_Mic_Sim_interRicci_nonneg} is different from the ABP approach in \cite{Bre22}, although the structures of the proofs bear a close resemblance.
    In our proof, the Kantorovich potential $\phi$ from optimal transport (see Theorem \ref{Intro_gen_McCann_thm}) plays the same role as the solution to a linear elliptic PDE in the ABP approach, and the key computation is the estimate of the Jacobian determinant of the map $\Phi$ along the geodesics $\gamma$ from Theorem \ref{Intro_gen_McCann_thm}, which is where the curvature condition is used (see Proposition \ref{Jacobian_est_along_geodesic_submfd} and Corollary \ref{Jacobian_est_submfd}; compared with Proposition 4.6 and Corollary 4.7 in \cite{Bre22}). 
    One can also see the connection between these two approaches from the definition of the $c$-superdifferential where a fixed optimal transport plan is supported: for  $\bar{x}\in \Sigma$, we have $\zeta \in \pt^{c_+}\phi(\bar{x})$ if and only if the function 
    \begin{equation*}
        \frac{1}{2}d^2_M(\cdot,\zeta)-\phi        
    \end{equation*}
    on $\Sigma$ is minimized at $\bar{x}$; this exact statement appeared in the argument of Lemma 4.2 in \cite{Bre22}.
    \par
    In the literature, both  optimal transport and the ABP method have been applied to prove many geometric and functional inequalities.
    For the optimal transport aspect, see e.g.
    D. Cordero-Erausquin, McCann, and M. Schmuckenschl\"ager's \cite{CMS01},
    Cordero-Erausquin, B. Nazaret, and C. Villani's \cite{CNV04},
    A. Figalli, F. Maggi, and A. Pratelli's \cite{FMP10}.
    For the ABP method, see e.g. N. Trudinger's \cite{Tru94}, 
    X. Cabr\'e's \cite{Cab08},
    Y. Wang and X. Zhang's \cite{WZ13},
    Cabr\'e, X. Ros-Oton, and J. Serra's \cite{CRS16}, 
    C. Xia and X. Zhang's \cite{XZ17}.
    We remark that optimal transport also applies to more general metric measure spaces, e.g. F. Cavalletti and  Mondino's \cite{CM17}, G. Antonelli, E. Pasqualetto, M. Pozzetta, and D. Semola's \cite{APPS22}, Cavalletti and D. Manini's \cite{CM22}.
    Also see \cite{AFM20}, \cite{FM22} for other approaches.
    \par
    Among these results about inequalities, we point out that a sharp Sobolev inequality for domains in complete manifolds with nonnegative Ricci curvature was proved also in \cite{Bre22} by the ABP method, and by optimal transport in Z. Balogh and A. Krist\'aly's \cite{BK22}. 
    One goal of this article is to provide an optimal transport counterpart regarding the proof of the Michael-Simon-Sobolev inequalities for submanifolds, as in \cite{BE22} when the ambient space is Euclidean.

    \subsubsection{Positive lower bounds on Intermediate Ricci Curvatures}
    We apply our optimal transport argument to the case where the intermediate Ricci curvatures are bounded from below by some positive constants.
    The main result is the following:
    \begin{theorem}\label{Intro_interRicci_pos_tubular}
Let $M$ be a closed Riemannian manifold of dimension $n+m$ with $\Ric_n\geq (n-1)k_1$ and $\Ric_m \geq (m-1)k_2$ for some $k_1, k_2>0$ and $m\geq2$.
Let $\Sigma\subset M$ be a compact submanifold of dimension $n$, possibly with boundary $\pt \Sigma$.
    Let $H$ be the mean curvature vector of $\Sigma$. 
    Let $f$ be a positive smooth function on $\Sigma$.
For $\epsilon>0$, let $N_\epsilon$ be the $\epsilon$-tubular neighborhood of $\Sigma$.
Then we have
\begin{equation*}
\begin{split}
        &\left( \frac{\vol_M(M\backslash N_\epsilon)}{|B^m|\diam (M)^m\sinc^m\left( \epsilon\sqrt{\frac{k_2(m-1)}{m}}\right)}\right)^\frac{1}{n} \left(\int_\Sigma f^\frac{n}{n-1}\right)^\frac{n-1}{n}\\
        &\leq \cos\left(\epsilon\sqrt{\frac{k_1(n-1)}{n}}\right)\int_\Sigma f+\frac{\diam (M)\sinc\left(\epsilon\sqrt{\frac{k_1(n-1)}{n}}\right)}{n} \left(\int_{\pt \Sigma}f+\int_\Sigma |\nabla^\Sigma f|+\int_\Sigma f|H|\right). 
\end{split}
    \end{equation*}
\end{theorem}
The lower bounds here on $\Ric_n$ and $\Ric_m$ are written in this way since $\Sec\geq k>0$ implies $\Ric_n\geq (n-1)k$ and $\Ric_m\geq (m-1)k$ in general. 
We also remark that these two curvature bounds are assumed exactly to fit in our argument; they may implicitly imply one another (see Remark \ref{Rmk_interRicci}). 
Additionally, the argument can be modified for the case of hypersurfaces and give a similar inequality.
\par
Theorem \ref{Intro_interRicci_pos_tubular} potentially gives a lower bound on the volume of the tubular neighborhood $N_\epsilon$.
On the other hand, by taking $\epsilon\to 0$, we have the following Michael-Simon-Sobolev inequality for the submanifold $\Sigma$:
\begin{corollary}\label{Intro_Mic_Sim_M_closed_int_Ric_nonneg}
Let $M$ be a closed Riemannian manifold of dimension $n+m$ with  $\Ric_n,\Ric_m > 0$.
    Let $\Sigma\subset M$ be a compact submanifold of dimension $n$, possibly with boundary $\pt \Sigma$.
    Let $H$ be the mean curvature vector of $\Sigma$. 
    Let $f$ be a positive smooth function on $\Sigma$.
    Then we have
    \begin{equation*}
        \left( \frac{\vol_M(M)}{|B^m|\diam(M)^m}\right)^\frac{1}{n} \left(\int_\Sigma f^\frac{n}{n-1}\right)^\frac{n-1}{n}
        \leq \int_\Sigma f+\frac{\diam(M)}{n}\left(\int_{\pt \Sigma}f+\int_\Sigma |\nabla^\Sigma f|+\int_\Sigma f|H|\right).
    \end{equation*}
\end{corollary}
    We remark that Corollary \ref{Intro_Mic_Sim_M_closed_int_Ric_nonneg} also follows from the proof of Theorem \ref{Intro_Mic_Sim_interRicci_nonneg} by assuming $M$ is closed, and hence also holds when $\Ric_n,\Ric_m \geq 0$ (see the discussion in the beginning of Subsection \ref{Subsec_pos_inter_Ricci}).

    \subsubsection{Negative lower bounds on Intermediate Ricci Curvatures}

In the general case of negative lower bounds, we have the following local result:
\begin{theorem}\label{Intro_interRicci_neg_local}
Let $M$ be a complete Riemannian manifold of dimension $n+m$ with $\Ric_n\geq nk_1$ and $\Ric_m \geq mk_2$ for some $k_1, k_2<0$.
Let $\Sigma\subset M$ be a compact submanifold of dimension $n$, possibly with boundary $\pt \Sigma$.
    Let $H$ be the mean curvature vector of $\Sigma$. 
    Let $f$ be a positive smooth function on $\Sigma$.
Assume that $\Sigma$ is contained in some geodesic ball $B^M_{\frac{r}{2}}(x_0)$.
Then we have
\begin{equation*}
\begin{split}
        &\left( \frac{\vol_M\left(B^M_{\frac{r}{2}}(x_0)\right)(-k_2)^\frac{m}{2}}{|B^m|\sinh^m(r\sqrt{-k_2})}\right)^\frac{1}{n} \left(\int_\Sigma f^\frac{n}{n-1}\right)^\frac{n-1}{n}\\
        &\leq \cosh(r\sqrt{-k_1})\int_\Sigma f
        +\frac{\sinh(r\sqrt{-k_1})}{n\sqrt{-k_1}}\left(\int_{\pt \Sigma}f+\int_\Sigma |\nabla^\Sigma f|+\int_\Sigma f|H|\right).    
\end{split}
    \end{equation*}
\end{theorem}
The lower bounds here on $\Ric_n$ and $\Ric_m$ are written in this way since $\Sec\geq k$ for some $k<0$ implies $\Ric_n\geq nk$ and $\Ric_m\geq mk$ in general. 
We also remark that these two curvature bounds are assumed exactly to fit in our argument; they may implicitly imply one another (see Remark \ref{Rmk_interRicci}). 
\par
Theorem \ref{Intro_interRicci_neg_local} has some local feature as the inequality might become trivial as $r\to \infty$, e.g.
in the $(n+m)$-hyperbolic space $\bb{H}^{n+m}$ with $k_1=-1=k_2$.
Recently, J. Cui and P. Zhao \cite{CZ22} proved a sharp Michael-Simon-Sobolev inequality for star-shaped hypersurfaces in the hyperbolic space.

\subsection{Outline}
    The outline of this article is as follows.
    In Section \ref{Sec_prelim}, we set up the terminology and summarize some classical results in optimal transport theory.
    In Section \ref{Sec_OT_Submfd}, we prove the results about optimal transport in the submanifold setting, including the generalization of McCann's theorem and the change of variable formula.
    In Section \ref{Sec_Michael_Simon}, we prove the Michael-Simon inequalities for compact submanifolds under the assumption that the ambient manifold has lower bounds of various signs on the intermediate Ricci curvatures.

\subsection*{Acknowledgment}
The author would like to thank Prof.\@ Aaron Naber for his guidance and inspiring comments.
He is also grateful to Prof.\@ Simon Brendle and Prof.\@ Andrea Mondino for their useful suggestions.
Additionally, he appreciates Prof.\@ Marco Pozzetta and Prof.\@ Gioacchino Antonelli's generosity in sharing various relevant references.
    
    \section{Preliminaries}\label{Sec_prelim}
    In this section, we set up the terminology and summarize some classical results in optimal transport theory, including McCann's theorem about optimal transport problems on Riemannian manifolds with the distance-squared cost.
    The contents here can be found in \cite{AG13}, \cite{McC01} and C. Villani's \cite{Vil09}.
    \par
        Let $X$ and $Y$ be Polish spaces equipped with Borel probability measures $\mu$ and $\nu$, respectively.
        Let $c:X\times Y\to \bb{R}$ be a measurable function.
        Viewing $c$ as the cost function, we would like to minimize the total cost to push $\mu$ forward to $\nu$.
        That is, we want to minimize
        \begin{equation}\label{OT_monge}
            \int_X c(x, T(x)) d\mu(x)
        \end{equation}
        among measurable maps $T:X\to Y$ with the push-forward condition $T_\sharp \mu=\nu$, meaning that $\mu(T^{-1}(B))=\nu(B)$ for any Borel set $B\subset Y$.
        This problem is the so-called Monge's formulation of optimal transport, and such minimizing maps are called \textit{optimal transport maps}.
        \par
        In general, Monge's formulation can be ill-posed; 
        for example, there can even be no such a map $T$ that pushes $\mu$ forward to $\nu$.
        Therefore, we may relax the problem to the following version, known as Kantorovich's formulation:
        we minimized
        \begin{equation}\label{OT_kanto}
            \int_{X\times Y} c(x,y)d \pi(x,y)
        \end{equation}
        among Borel probability measures $\pi$ on $X\times Y$ with the transport conditions $\pi\big |_X=\mu$ and $\pi\big |_Y=\nu$; such measures are called transport plans.
        Kantorovich's formulation is a relaxation since any transport map $T$ gives rise to a transport plan on $X\times Y$ by pushing $\mu$ forward via the map $\t{id}_X\times T:X\to X\times Y $, and it solves the issue above in Monge's formulation since transport plans always exist, e.g.\@ the product measure $\mu\times \nu$.
        Moreover, the following theorem gives the existence of minimizing measures under mild assumptions:
        \begin{theorem}[\cite{AG13}, Theorem 2.5]\label{exist_OT_plan}
        Assume that the cost function $c$ is lower semicontinuous and bounded from below.
        Then there exists a minimizer for Kantorovich's formulation (\ref{OT_kanto}).
        \end{theorem}
        Such minimizing measures are called \textit{optimal transport plans}.
        With stronger assumptions, it can be shown that such measures are supported on the graph of set-valued functions with generalized convexity (Theorem \ref{funda_OT} below); the precise statement requires the following terminologies associated to the cost function $c$.
        We denote by $\bar{\bb{R}}$ the extended real numbers $\bb{R}\cup \{\pm \infty\}$.
        \begin{definition}[$c_+$-transform]
            Let $\psi:Y\to \bar{\bb{R}}$ be any function.
            Its \textit{$c_+$-transform} $\psi^{c_+}:X\to \bar{\bb{R}}$ is defined as
            \begin{equation*}
                \psi^{c_+}(x)=\inf_{y\in Y} c(x,y)-\psi(y).
            \end{equation*}
            Similarly, given $\phi:X\to \bar{\bb{R}}$ we can define $\phi^{c_+}:Y\to \bar{\bb{R}}$ as
            \begin{equation*}
                \phi^{c_+}(y)=\inf_{x\in X} c(x,y)-\phi(x).
            \end{equation*}
        \end{definition}
        \begin{definition}[$c$-concavity]
            A function $
            \phi:X\to \bar{\bb{R}}$ is \textit{$c$-concave} if $\phi=\psi^{c_+}$ for some $\psi:Y\to \bar{\bb{R}}$.
            Similarly, a function $
            \psi:Y\to \bar{\bb{R}}$ is $c$-concave if $\psi=\phi^{c_+}$ for some $\phi:X\to \bar{\bb{R}}$.
        \end{definition}
            A simple observation is that $\psi^{c_+}=\psi^{c_+c_+c_+}$.
            Thus we have the following lemma:
            \begin{lemma}
            $\phi$ is $c$-concave if and only if $\phi=\phi^{c_+c_+}$.
            \end{lemma}
            
        Next we define the following graph-like set:
        \begin{definition}[$c$-superdifferential]
        Let $\phi:X\to \bar{\bb{R}}$ be any function.
        The \textit{$c$-superdifferential} $\pt^{c_+}\phi\subset X\times Y$ is defined as
        \begin{equation*}
            \pt^{c_+}\phi =\left\{(x,y)\in X\times Y\mid \phi(x)+\phi^{c_+}(y)=c(x,y)\right\}.
        \end{equation*}
        The \textit{$c$-superdifferential of $\phi$ at $x$}, denoted by $\pt^{c_+}\phi(x)$, is the set of $y\in Y$ such that $(x,y)\in \pt^{c_+}\phi$.
            Similarly, we can define $c$-superdifferential of $\psi:Y\to \bar{\bb{R}}$.
        \end{definition}
        \begin{remark}
            Since the cost function $c$ takes value in $\bb{R}$, we have that $(x,y)\in \pt^{c_+}\phi$ only if both $\phi(x)$ and $\phi^{c_+}(y)$ are finite.
        \end{remark}
        \begin{remark}
            Since $\phi^{c_+}(y)\leq c(z,y)-\phi(z)$ for all $z\in X$, we have that $(x,y)\in \pt^{c_+}\phi$ if and only if the function $c(\cdot, y)-\phi$ on $X$ is minimized at $x$.
            We also remark that if $\phi$ is $c$-concave, then $\pt^{c_+}\phi=\pt^{c_+}\phi^{c_+}$.
        \end{remark}

        With these terminologies introduced, we are now ready to state the following classical theorem:
        \begin{theorem}[\cite{AG13}, Theorem 2.13; also see \cite{Rus96}, Theorem 2.2]\label{funda_OT}
        Assume that the cost function $c:X\times Y\to \bb{R}$ is continuous and bounded from below, and there are functions $a\in L^1(\mu)$, $b\in L^1(\nu)$ such that
        \begin{equation*}
            c(x,y)\leq a(x)+b(y).            
        \end{equation*}
        Let $\pi$ be a Borel probability measures on $X\times Y$ with $\pi\big |_X=\mu$ and $\pi\big |_Y=\nu$. 
        Then $\pi$ is optimal if and only if there exists a $c$-concave function $\phi$ on $X$ such that $\max \{\phi,0\}\in L^1(\mu)$ and $\supp\pi\subset \pt^{c_+}\phi$.
        \end{theorem}
        \begin{remark}
            Such a $c$-concave function $\phi$ is sometimes called a \textit{Kantorovich potential}.
        \end{remark}
        \par
        Now we specify the two spaces $X$ and $Y$ to be compact subsets of a complete Riemannian manifold $M$ and let $c(x,y)=\frac{1}{2}d^2_M(x,y)$ on $X\times Y$ be the cost function, where $d_M$ is the geodesic distance function on $M$.
        Under this setting, McCann in \cite{McC01} showed the existence of optimal transport maps, generalizing Brenier's theorem \cite{Bre91} in the Euclidean setting:
            \begin{theorem}[McCann's theorem, \cite{McC01}]\label{McCann_thm}
        Let $\mu$ and $\nu$ be Borel probability measures on a compact regular domain $X$ and a compact subset $Y$ in a complete Riemannian manifold $M$, respectively.
        Suppose $\mu$ is absolutely continuous with respect to the volume measure.
        Then there is an optimal transport map $F$ pushing $\mu$ forward to $\nu$.
        Furthermore, $F$ is given $\mu$-a.e.\@ by $F(x)=\exp_x(-\nabla \phi(x))$ for some $c$-concave function $\phi$ on $X$, and the geodesic $\gamma(t):=\exp_x(-t\nabla \phi(x))$ for $t\in [0,1]$ is minimizing.
    \end{theorem}
    \begin{proof}[Sketch proof]
    Since the cost function $c$ is continuous and bounded, Theorem \ref{exist_OT_plan} gives the  existence of an optimal transport plan $\pi$, and by Theorem \ref{funda_OT} we have $\supp \pi \subset \partial^{c_+}\phi$ for some $c$-concave function $\phi$ on $X$.
    Next, Lemma 2 in \cite{McC01} showed that $\phi$ is Lipschitz, and hence by Rademacher's theorem (see e.g. Lemma 4 in \cite{McC01}) $\phi$ is differentiable a.e.\@ w.r.t.\@ the volume measure $\vol_M$.
    Since $\mu\ll \vol_M$, $\phi$ is thus differentiable $\mu$-a.e. 
    On the other hand, Lemma 7 in \cite{McC01} showed that $\pt^{c_+}\phi(x)$ contains only the point $\exp_x(-\nabla \phi(x))$ whenever $\phi$ is differentiable at $x$, and in this case the geodesic $\gamma$ is minimizing.
    It is now direct to check that the map $F(x)=\exp_x(-\nabla \phi(x))$ defined for $\mu$-a.e.\@ $x\in X$ is an optimal transport map.
    \end{proof}
    \begin{remark}
        \cite{McC01} also proved the uniqueness of the optimal transport map.
        In addition, the assumptions of Theorem \ref{McCann_thm} can be relaxed; see \cite{GM96}, \cite{Gig11}.
    \end{remark}

    \section{Optimal Transport from  measures on submanifolds}\label{Sec_OT_Submfd}
 
    In this section, we prove a generalization of McCann's theorem in a submanifold setting, which says that for a.e.\@ $x\in \Sigma$, the $c$-differential $\pt^{c_+}\phi(x)$ can roughly be identified as some specific part of the normal fiber $T^\perp_x \Sigma$.
    Together with the transport condition, we deduce a change of variable formula.
    The arguments here are inspired by \cite{Bre22}, \cite{McC01}, and the details involve results from \cite{CMS01}, \cite{Vil09}, F. Santambrogio's \cite{San15}, Ambrosio, N. Fusco, and D. Pallara's \cite{AFP00}.
    \par
    
    The setting is as follows.
    Let $M$ be a complete Riemannian manifold of dimension $n+m$. 
    Let $\Omega\subset M$ be a compact regular domain and $\Sigma\subset M$ a compact submanifold of dimension $n$, possibly with boundary.
    Let $\mu$ and $\nu$ be Borel probability measures on $\Sigma$ and $\Omega$, respectively, and we assume they are absolutely continuous w.r.t.\@ the corresponding volume measures.
    \par
    \textbf{Notation}: The operators and maps associated to $\Sigma$ will be subscripted or superscripted by $\Sigma$, e.g.,$\nabla^\Sigma, \exp^\Sigma, \Delta_\Sigma$.
    For $x\in \Sigma$ and $v\in T_x M$, we write $v=v^T+v^\perp$ as the orthogonal decomposition of $v$, where $v^T\in T_x\Sigma$ and $v^\perp \in T^\perp_x \Sigma$.
    We will often write $d$ instead of $d_M$ to denote the geodesic distance function on $M$, and $d_\zeta(\cdot)$ denotes the distance from some fixed point $\zeta\in M$.
    \par
    Consider the optimal transport problem between $\mu$ and $\nu$ with $c=\frac{1}{2}d^2$ on $\Sigma \times \Omega$ as the cost function.
    Since $c$ is continuous and bounded, Theorem \ref{exist_OT_plan} applies to give an optimal transport plan $\pi$ as a measure on $\Sigma \times \Omega$, and Theorem \ref{funda_OT} shows that $\supp \pi \subset \partial^{c_+}\phi$ for some $c$-concave function $\phi$ on $\Sigma$.
    By Lemma 2 in \cite{McC01} and the fact that $d_M$ is dominated by the intrinsic distance on $\Sigma$, $\phi$ is then Lipschitz on $\Sigma$.
    Thus by Rademacher's theorem (see e.g. \cite{McC01}, Lemma 4), $\phi$ is differentiable a.e.\@ 
    The following results provide the description we need about $\pt^{c_+}\phi(x)$ whenever $\phi$ is differentiable at $x$; they can be viewed as the submanifold version of Proposition 6 and Lemma 7 in \cite{McC01}.
    
    \begin{proposition}[Superdifferentiability of distance squared restricted to submanifolds]\label{supdiff_d^2_submfd}
        Let $x\in \Sigma\backslash \pt \Sigma, \zeta\in M$, and let $\sigma:[0,1]\to M$ be a minimizing geodesic in $M$ from $\zeta$ to $x$.
        Then the function $\frac{1}{2}d^2_{\zeta}\big |_\Sigma$ on $\Sigma$ has a supergradient $\dot{\sigma}(1)^T$ at $x$, in the sense that
    \begin{equation*}
        \frac{1}{2}d^2_{\zeta}\big |_\Sigma(\exp_x u)\leq \frac{1}{2}d^2_{\zeta}\big |_\Sigma(x)+\langle \dot{\sigma}(1)^T,u\rangle+o(|u|)        
    \end{equation*}
    for any $u\in T_x \Sigma$.
    \end{proposition}

    \begin{proof}
        The proof is similar to that of Proposition 6 in \cite{McC01}. 
        Choose $\epsilon>0$ and a neighborhood $V$ of $x$ in $M$ such that for any $\eta\in V$, $\exp^M_\eta$ maps $B_\epsilon(0)\subset T_\eta M$ diffeomorphically to some $U_\eta\supset V$.
        We first prove the case when $\zeta\in V$.
        The proof follows from the following computation: for $u\in T_x\Sigma$ with small norm,
        \begin{align*}
            \frac{1}{2}d^2_\zeta(\exp^\Sigma_x u)
            &=\frac{1}{2}d^2\left(\zeta, \exp^M_\zeta (\exp_\zeta^M)^{-1}\exp^\Sigma_x u\right)\\
            &=\frac{1}{2}\left|(\exp^M_\zeta)^{-1} \exp^\Sigma_x u\right|_\zeta^2\\
            &=\frac{1}{2}\left|\dot{\sigma}(0)+D_x((\exp^M_\zeta)^{-1})D_0(\exp^\Sigma_x)u+o(|u|)\right|_\zeta^2\quad \t{(linearization at $u=0$)}\\
            &=\frac{1}{2}|\dot{\sigma}(0)|^2_\zeta+\left\langle \dot{\sigma}(0), (D_{\dot{\sigma}(0)}\exp^M_\zeta)^{-1}u\right\rangle + o(|u|)\\
            &=\frac{1}{2}d_\zeta^2(x)+\langle \dot{\sigma}(1)^T, u\rangle + o(|u|),
        \end{align*}
        where we use Gauss' lemma in the last step.
        Thus we have $\nabla^\Sigma|_x (\frac{1}{2}d^2_\zeta\big|_\Sigma)=\dot{\sigma}(1)^T$.
        Applying the chain rule yields $\nabla^\Sigma|_x \left(d_\zeta\big |_\Sigma\right)=\frac{\dot{\sigma}(1)^T}{|\dot{\sigma}(1)|}$ (if $\zeta\neq x$).
        \par
        For $\zeta\not\in V$, choose a point $\eta$ on the geodesic $\sigma$ such that $\eta\in V\backslash\{x\}$.
        From the result above, we have $\nabla^\Sigma|_x \left(d_\eta\big|_\Sigma\right)=\frac{\dot{\sigma}(1)^T}{|\dot{\sigma}(1)|}$, and hence
        \begin{align*}
            d_\zeta(\exp^\Sigma_x u)&\leq d(\zeta,\eta)+d_\eta(\exp^\Sigma_x u)\\
            &=d(\zeta,\eta)+d_\eta(x)+\left\langle\frac{\dot{\sigma}(1)^T}{|\dot{\sigma}(1)|},u\right\rangle+ o(|u|)\\
            &=d_\zeta(x)+\left\langle\frac{\dot{\sigma}(1)^T}{|\dot{\sigma}(1)|},u\right\rangle+ o(|u|).
        \end{align*}
        Therefore $d_\zeta\big |_\Sigma$ has supergradient $\frac{\dot{\sigma}(1)^T}{|\dot{\sigma}(1)|}$ at $x$.
        Applying the one-sided chain rule (\cite{McC01}, Lemma 5) yields that $\frac{1}{2}d_\zeta^2\big|_\Sigma$ has supergradient $d_\zeta(x)\frac{\dot{\sigma}(1)^T}{|\dot{\sigma}(1)|}=\dot{\sigma}(1)^T$ at $x$.
    \end{proof}   
    \begin{lemma}[Tangency to submanifolds]\label{Tangency_submfd}
        Let $x\in \Sigma\backslash \pt \Sigma$ and $\zeta\in \pt^{c_+}\phi(x)$.
        If $\phi$ is differentiable at $x$, then $\zeta=\exp^M_x(-\nabla^\Sigma \phi(x)+v)$ for some $v\in T^\perp_x \Sigma$.
        In addition, the curve $\gamma(t):=\exp^M_x(-t\nabla^\Sigma \phi(x)+tv)$ for $t\in [0,1]$ is a minimizing geodesic from $x$ to $\zeta$.
    \end{lemma}  
\begin{proof}
    Since $\zeta\in \pt^{c_+}\phi(x)$, 
    for any $u\in T_x\Sigma$ with small norm, we have
    \begin{align*}
        \frac{1}{2}d^2_\zeta(\exp^\Sigma_x u)&\geq\frac{1}{2}d_\zeta^2(x)-\phi(x)+\phi(\exp^\Sigma_x u)\\
        &=\frac{1}{2}d_\zeta^2(x)-\phi(x)+\phi(x)+\langle u, \nabla^\Sigma \phi(x)\rangle+o(|u|).
    \end{align*}
    That is to say, the function $\frac{1}{2}d^2_\zeta\big |_\Sigma$ has a subgradient $\nabla^\Sigma \phi(x)$ at $x$.
    On the other hand, let $\sigma:[0,1]\to M$ be a minimizing geodesic from $\zeta$ to $x$.
    By Proposition \ref{supdiff_d^2_submfd} above, $\dot{\sigma}(1)^\top$ is a supergradient of $\frac{1}{2}d^2_\zeta\big |_\Sigma$ at $x$.
    Thus $\frac{1}{2}d^2_\zeta\big |_\Sigma$ is actually differentiable at $x$ and $\dot{\sigma}(1)^\top=\nabla^\Sigma \phi(x)$.
    Therefore $\dot{\sigma}(1)=\nabla^\Sigma \phi(x)-v$ for some $v\in T^\perp_x\Sigma$.
    By reversing the direction of $\sigma$, we conclude
    \begin{equation*}
        \zeta=\exp^M_x(-\dot{\sigma}(1))=\exp^M_x(-\nabla^\Sigma \phi(x)+v).        
    \end{equation*}
\end{proof}
    
    Lemma \ref{Tangency_submfd} above suggests that we consider the following map: 
    \begin{equation}
    \Phi(x,v):=\exp^M_x(-\nabla^\Sigma \phi(x)+v)        
    \end{equation}
for $x\in \Sigma$ and $v\in T_x^\perp \Sigma$; this is well-defined for a.e.\@ $x\in \Sigma$ since $\phi$ is differentiable a.e., and the image of $\Phi$ then contains $\pt^{c_+}\phi(x)$ for such $x$.
We here remark that since $\Omega$ is compact, a maximizer argument as in Lemma 7 in \cite{McC01} shows that $\pt^{c_+}\phi(x)$ is never empty.
\par
Now we show that there is a measurable subset $A$ of $T^\perp\Sigma$ such that $\Phi$ becomes an injection when restricted to $A$; additionally, $A$ projects to a full $\mu$-subset in $\Sigma$, and $\Phi(A)$ has full $\nu$-measure in $\Omega$.
Our argument here involves the optimal transport problem from $(\Omega, \nu)$ to $(\Sigma, \mu)$ as in the classical McCann's theorem (Theorem\ref{McCann_thm}), which gives 
\begin{equation}\label{OT_map_nu_to_mu}
    F(\zeta):=\exp^M_\zeta (-\nabla^M \phi^{c_+}(\zeta))
\end{equation}
as an optimal transport map from $\nu$ to $\mu$ (recall that $\phi^{c_+}$ here is a function on $\Omega$).
We also need a technical regularity result that $c$-concave functions are semiconcave, in the sense that locally they differ from being geodesically concave by some smooth functions.
We refer to Definition 3.8 in \cite{CMS01} for the precise definition; the proof of this fact can be found in Proposition 3.14 in \cite{CMS01} and Proposition \ref{submfd_c_concave_semiconcave}.
Hence by Alexandrov-Bangert theorem (Theorem 3.10 in \cite{CMS01}; also see Theorem 14.1 in \cite{Vil09}), the \textit{Alexandrov Hessians}   $\Hess^\Sigma\phi$ and $\Hess \phi^{c_+}$ exist a.e., and they define measurable sections of bilinear forms on the tangent bundles.

\begin{theorem}[Generalized McCann's theorem]\label{gen_McCann_thm}
There is a measurable subset $A\subset T^{\perp}\Sigma$ such that 
\begin{enumerate}
    \item $\Phi$ is well-defined and injective on $A$;
    \item  $p_{T^\perp \Sigma}(A)$ is measurable and $\mu(p_{T^\perp \Sigma}(A))=1$;
    \item $\Phi(A)$ is measurable and $\nu(\Phi(A))=1$.
\end{enumerate}
Here, $p_{T^\perp \Sigma}:T^\perp \Sigma \to \Sigma$ is the bundle projection map.
\end{theorem}
\begin{proof}
By Alexandrov-Bangert theorem and that $\mu,\nu$ are absolute continuous, we can choose Borel subsets $K\subset \Sigma$ and $L\subset \Omega$ such that $\Hess^\Sigma \phi$ exists on $K$, $\Hess \phi^{c_+}$ exists on $L$, and $\mu(K)=1=\nu(L)$.
Then we have $\nu(F^{-1}(K))=\mu(K)=1$ from the transport condition.
On the other hand, from McCann's theorem (Theorem \ref{McCann_thm}) we can write $F(L)=p_{\Sigma\times \Omega\to \Sigma}(\pt^{c_+}\phi^{c_+}\cap (\Sigma\times L))$,
where $p_{\Sigma\times \Omega\to \Sigma}$ is the canonical projection map.
Hence  by the theory of analytic sets (see e.g. Proposition 8.4.4 in \cite{Coh13}), $F(L)$ is measurable.
We can now use the transport condition (after extending to Lebesgue measurable subsets) to obtain
\begin{equation*}
    \mu(\Sigma \backslash F(L))=\nu(F^{-1}(\Sigma \backslash F(L)))\leq \nu(\Omega\backslash L)=0,
\end{equation*}
and thus $\mu(F(L))=1$.
We hereby conclude that $\mu(K\cap F(L))=1=\nu(F^{-1}(K)\cap L)$.
\par
Fix $\zeta\in F^{-1}(K)\cap L$.
Since $\Hess \phi^{c_+}(\zeta)$ exists, by Proposition 4.1 in \cite{CMS01} $F(\zeta)$ is not in the cut locus of $\zeta$.
Thus there is a unique minimizing geodesic $\sigma:[0,1]\to M$ from $\zeta$ to $F(\zeta)$, and $\dot{\sigma}(1)=\nabla(\frac{1}{2}d_\zeta^2)_{F(\zeta)}$.
This shows that the following map:
\begin{equation}\label{map_Theta}
    \Theta(\zeta):=\left(F(\zeta), -\nabla\left(\frac{1}{2}d_\zeta^2\right)_{F(\zeta)}^\perp\right)    
\end{equation}
is well-defined on $F^{-1}(K)\cap L$.
On the other hand, by Lemma 7 in \cite{McC01}, we have $F(\zeta)\in \pt^{c_+}\phi^{c_+}(\zeta)$, which is equivalent to $\zeta\in \pt^{c_+}\phi(F(\zeta))$.
Since $F(\zeta)\in K$, $\phi$ is differentiable at $F(\zeta)$.
Thus by Lemma \ref{Tangency_submfd} we have $\dot{\sigma}(1)^\top=\nabla^\Sigma \phi(F(\zeta))$.
In summary, we have
\begin{equation*}
    \Phi\circ \Theta(\zeta)=\exp_{F(\zeta)}\left(-\nabla^\Sigma \phi(F(\zeta))-\nabla\left(\frac{1}{2}d_\zeta^2\right)_{F(\zeta)}^\perp\right)
    =\exp_{F(\zeta)}(-\dot{\sigma}(1))=\zeta.
\end{equation*}
That is, $\Phi\circ\Theta$ is the identity map on $F^{-1}(K)\cap L$.
We now define
\begin{equation}
     A:=\Theta(F^{-1}(K)\cap L).
\end{equation}
Clearly $\Phi$ is then well-defined and injective on $A$, and both $p_{T^\perp \Sigma}(A)=F(F^{-1}(K)\cap L)=K\cap F(L)$ and $\Phi(A)=F^{-1}(K)\cap L$ are measurable with full measures.
\par
To show that $A$ is measurable, we use the following characterization of $A$.
Consider the following subsets:
\begin{align*}
    U&:=\{(x,v)\in T^\perp \Sigma|_K \mid \gamma(t):=\exp^M_x(-t\nabla^\Sigma \phi(x)+tv)\, \t{is minimizing for $t\in[0,1]$}\};\\
    V&:=\{(x,v)\in T^\perp \Sigma|_K\mid \Phi(x,v)\in \pt^{c_+}\phi(x)\}.
\end{align*}
Clearly $U$ and $V$ are measurable, and it is direct to check that $    U\cap V\cap \Phi^{-1}(F^{-1}(K)\cap L)=A$.
\end{proof}
\begin{remark}
    In the above proof of Theorem $\ref{gen_McCann_thm}$, it suffices to assume that $\phi$ is merely differentiable on $K$ instead of the existence of $\Hess^\Sigma \phi$.
    However, the existence of the Hessian is needed for the proofs of Michael-Simon-Sobolev inequalities in Section 
    \ref{Sec_Michael_Simon}.
\end{remark}
Theorem \ref{gen_McCann_thm} above has the following corollary:
\begin{corollary}[$c$-superdifferential and normal fiber]\label{superdiff_normalfib}
Define $A_x:=A\cap T^\perp_x \Sigma$ for $x\in p_{T^\perp \Sigma}(A)$.
Then the restricted map
\begin{equation*}
\Phi(x,\cdot):A_x\to \pt^{c_+}\phi(x)\cap \Phi(A)    
\end{equation*}
is a bijection.
\end{corollary}

With the results above, we now deduce a change of variable formula from the transport condition:

    \begin{theorem}\label{change_of_var}
    Let $A$ be the subset in Theorem \ref{gen_McCann_thm}.
    Suppose $f$ and $g$ are the density functions of $\mu$ and $\nu$ w.r.t.\@ the volume measures of $\Sigma$ and $M$.
    Then for $\mu$-a.e.\@ $x\in \Sigma$,
    \begin{equation*}
    \int_{A_x}g\circ \Phi(x,v)|\det D\Phi(x,v)| dv=f(x).        
    \end{equation*}
    \end{theorem}
\begin{proof}
First of all, by the disintegration theorem (see e.g. \cite{AGS05}, Theorem 5.3.1) we have a family of probability measures $(\pi_x)$ on $\Omega$ parametrized by $\mu$-a.e.\@ $x\in \Sigma$ such that
\begin{equation*}
    \int_{\Sigma \times \Omega} h d\pi=\int_\Sigma \int_\Omega h(x,\zeta) d\pi_x(\zeta)d\mu(x)    
\end{equation*}
for any nonnegative measurable function $h$ on $\Sigma \times \Omega$.
Next, we already know that $\supp \pi \subset\pt^{c_+} \phi$.
Furthermore by Theorem  \ref{gen_McCann_thm}, $\Phi(A)$ has full measure in $\Omega$, and by Corollary \ref{superdiff_normalfib}we have the bijection via $\Phi$ between $A_x$ and $\pt^{c_+}\phi(x)\cap \Phi(A)$ for $\mu$-a.e.\@ $x$ .
Thus if $h$ depends only on $\zeta\in\Omega$, then we can write
\begin{equation}\label{change_of_var_pf_LHS}
    \begin{split}
    \int_{\Omega} h d\nu
    &=\int_\Sigma \int_{\pt^{c_+} \phi(x) \cap \Phi(A)} h(\zeta) d\pi_x(\zeta)d\mu(x)\\
    &=\int_\Sigma f(x)\int_{A_x} h\circ \Phi(x,v) d\lambda_x(v)d\vol_\Sigma(x),
    \end{split}
\end{equation}
    where $\lambda_x$ is the push-forward probability measure of $\pi_x$ by the map $\Theta$ in Theorem \ref{gen_McCann_thm}.
    \par
    On the other hand, since $\phi$ is semi-convex, we know that $\nabla^\Sigma \phi$ is countably Lipschitz (see Theorem 5.34 in \cite{AFP00}, or Sec 1.7.6. in \cite{San15}), and hence so is $\Phi$.
    Thus from the area formula of the map $\Phi$ and that $\Phi(A), p_{T^\perp \Sigma}(A)$ have full measures, we get
    \begin{equation}\label{change_of_var_pf_RHS}
        \begin{split}
        \int_\Omega hd\nu
        &=\int_\Omega hgd\vol_M\\
        &=\int_\Sigma \int_{A_x} h\circ \Phi(x,v) g\circ \Phi(x,v) |\det D\Phi(x,v)|dv d\vol_\Sigma (x).
        \end{split}
    \end{equation}
    \par
    Now let $a(x)$ be any nonnegative measurable function on $\Sigma$, and we plug $h=a\circ F$ into equalities (\ref{change_of_var_pf_LHS}) and (\ref{change_of_var_pf_RHS}) above.
    Then since $h\circ \Phi(x,v)=a(x)$ from Theorem \ref{gen_McCann_thm}, we have
    \begin{equation*}
         \int_\Sigma f(x)a(x)\int_{A_x}  d\lambda_x(v)d\vol_\Sigma(x)
         =\int_\Sigma a(x)\int_{A_x}  g\circ \Phi(x,v) |\det D\Phi(x,v)|dv d\vol_\Sigma (x).        
    \end{equation*}
    The theorem now follows since $a$ is arbitrary and $\lambda_x$ is a probability measure.
\end{proof}

\section{Proofs of Michael-Simon-Sobolev inequalities}\label{Sec_Michael_Simon}
In this section, we apply our generalization of McCann's theorem to prove Michael-Simon-Sobolev inequalities for compact submanifolds in complete manifolds with intermediate Ricci curvature bounded from below.
We here recall the following definitions from \cite{KM18}:
\begin{definition}[$p$-Ricci Curvature]
For a $p$-dimensional plane $P$ in $T_{x}M$ and a vector $w\in T_{x}M$, we define the \textit{$p$-Ricci curvature} of $P$ in the direction of $w$ as
\begin{equation*}
\Ric_p(P,w):=\tr\left[\top_{P}\circ(R(w,\cdot)w \big)|_P\right],    
\end{equation*}
where $\top_P:T_{x}M\rightarrow P$ is the orthogonal projection of $T_{x}M$ onto $P$, and $R$ is the curvature tensor of $M$.
\end{definition}
Notice that if we choose $e_1,\dots,e_p$ as an orthonormal basis of $P$, then we have
\begin{equation*}
    \Ric_p(P,w)=\sum_{i=1}^p \langle R(w,e_i)w, e_i\rangle.    
\end{equation*}

\begin{definition}[$p$-Ricci lower bounds]
We say that $M$ has \textit{$p$-Ricci curvature bounded from below by $K$} if for any $x \in M$, any $w\in T_xM$, and any $p$-dimensional plane $P\subset T_{x}M$, we have $\Ric_p(P,w)\geq K|w|^2$. In this case we write $\Ric_{p}\geq K$.
\end{definition}
\begin{remark}\label{Rmk_interRicci}
It is clear that $\Sec\geq 0$ implies $\Ric_p\geq0$ for all $p$.
More generally, we have (see \cite{KM18}, Remark 2.3 for details):
\begin{itemize}
    \item $\Sec\geq k\geq 0$ implies $\Ric_p\geq (p-1)k$ for all $1\leq p\leq \dim M$.
    \item $\Sec\geq k$ with $k<0$ implies $\Ric_p \geq pk$ for $1\leq p\leq \dim M-1$.
    \item $\Ric_2\geq k\geq0$ is equivalent to $\Sec\geq k\geq 0$.
\end{itemize}
We also remark that if $\Ric_p\geq K$ for some $2\leq p\leq \dim M$, then $\Ric_q\geq \frac{qK}{p}$ for any $q$ with $p\leq q\leq \dim M$.
\end{remark}

\begin{remark}
Some authors (e.g.\@ \cite{Cha20}) use different normalizations for lower bounds on the intermediate Ricci curvatures.
\end{remark}

\subsection{Manifolds with nonnegative intermediate Ricci curvatures}\label{Subsec_nonneg_inter_Ricci}
The goal of this section is to prove Theorem \ref{Intro_Mic_Sim_interRicci_nonneg}, and we here use the notations from Section \ref{Sec_OT_Submfd}.
In addition, we denote the second fundamental form of the submanifold $\Sigma$ by $\sff$, and $H=\tr \sff$ is the mean curvature vector.
\par
For parameters $0<\sigma<1$ and $r>0$ large enough, define the following annulus-like subset
\begin{equation*}
    \Omega=\{p\in M: \sigma r\leq d(x,p)\leq r, \forall x\in \Sigma\}.    
\end{equation*}
By approximation, we may assume that $\Omega$ is a regular domain.
Consider the optimal transport problem of the following two measures
\begin{equation*}
    \mu=\frac{1}{\int_\Sigma f^\frac{n}{n-1}}f^\frac{n}{n-1} \vol_\Sigma,\;
    \nu=\frac{1}{\vol(\Omega)}\vol_M\llcorner \Omega,
\end{equation*}
with the cost $c(x,\zeta)=\frac{1}{2}d(x,\zeta)^2$ on $\Sigma \times \Omega$.
By the generalized McCann's theorem (Theorem \ref{gen_McCann_thm}), we have a $c$-concave function $\phi$ on $\Sigma$, and the map 
\begin{equation*}
    \Phi(x,v):=\exp_x(-\nabla^\Sigma \phi(x)+v)    
\end{equation*}
defined on the subset $A\subset T^\perp \Sigma$ with the specific properties.
Note that these objects may depend implicitly on the parameters $\sigma,r$.
\par
The first step of the proof is the following inequality from integration by parts; this is also the last step in the arguments of \cite{BE22}: 
\begin{lemma}\label{int_by_part_submfd}
\begin{equation*}
    -\int_\Sigma f\Delta^{\ac}_\Sigma \phi 
    \leq r\int_{\pt \Sigma}f+r\int_\Sigma |\nabla^\Sigma f|,       
\end{equation*}
where $\Delta^{\ac}_\Sigma \phi$ is the trace of the Alexandrov Hessian of $\phi$.
\end{lemma}
    \begin{proof}
        Recall that since $\phi$ is $c$-concave, Proposition \ref{submfd_c_concave_semiconcave} yields that $\phi$ is semiconcave.
        Then by Alexandrov-Bangert theorem, the Alexandrov Hessian $\Hess^\Sigma \phi$ exists a.e.\@; additionally, the Lebesgue decomposition of the distributional Laplacian $[\Delta_\Sigma \phi]$ (as a measure) can be written as
        \begin{equation}\label{Leb_decomp_lap_phi}
            [\Delta_\Sigma \phi]=\Delta^{\ac}_\Sigma\phi \,\vol_\Sigma+[\Delta_\Sigma \phi]^s, 
        \end{equation}
       and the singular part $[\Delta_\Sigma \phi]^s$ is nonpositive.
        \par
        For $\epsilon>0$ small enough, let $\lambda_\epsilon$ be a smooth cut-off of $\Sigma$ w.r.t.\@ $\pt \Sigma$,
        i.e. $\lambda_\epsilon\geq0$, $\lambda_\epsilon \equiv 1$ on $\Sigma_\epsilon:=\{x\in \Sigma\mid d_\Sigma(x,\pt \Sigma)>\epsilon\}$ and $\lambda_\epsilon \in C^\infty_c(\Sigma \backslash \pt \Sigma)$.
        Using the decomposition (Line \ref{Leb_decomp_lap_phi}) above, we then have
        \begin{equation*}
            \int_\Sigma \Delta_\Sigma (\lambda_\epsilon f)\phi ~d\vol_\Sigma
            =\int_\Sigma \lambda_\epsilon f ~d[\Delta_\Sigma \phi]\leq \int_\Sigma \lambda_\epsilon f\Delta^{\ac}_\Sigma\phi ~d\vol_\Sigma.            
        \end{equation*}
        On the other hand, integrating by parts with $|\nabla^\Sigma \phi(x)|\leq r$ since $\Phi(x,v)\in \Omega$,  we have
        \begin{equation*}
            \begin{split}
            \int_\Sigma \Delta_\Sigma (\lambda_\epsilon f)\phi ~d\vol_\Sigma
            &=-\int_\Sigma \langle\nabla^\Sigma(\lambda_\epsilon f), \nabla^\Sigma \phi\rangle d\vol_\Sigma\\
            &\geq -r\int_\Sigma \lambda_\epsilon|\nabla^\Sigma f|d\vol_\Sigma
            -r\int_\Sigma f|\nabla^\Sigma\lambda_\epsilon|d\vol_\Sigma.                
            \end{split}
        \end{equation*}
    The lemma now follows by combining these two inequalities and taking the limit as $\epsilon\to 0$.
    \end{proof} 

The second step is to use the change of variable formula (Theorem \ref{change_of_var}), which gives
\begin{equation}\label{Jacobian_det_Phi}
    \int_{A_x}|\det D\Phi(x,v)|dv
    =\frac{\vol(\Omega)}{\int_\Sigma f^\frac{n}{n-1}}f^\frac{n}{n-1}(x)
\end{equation}
for $\mu$-a.e.\@ $x\in \Sigma$.
\par

The third step is to estimate $\det D\Phi(\bar{x}, \bar{v})$ for fixed $(\bar{x}, \bar{v})\in A$ by $\Delta^{\ac}_\Sigma \phi(\bar{x})$.
As in Proposition 4.6 of \cite{Bre22} and in Chapter 14 of \cite{Vil09}, we use the Jacobi fields from the geodesic variation
\begin{equation}\label{def_geo_var}
    \Phi_t(x,v)=\exp_x(-t\nabla^\Sigma \phi(x)+tv)    
\end{equation}
for $t\in [0,1]$ and denote $\bar{\gamma}(t):=\Phi_t(\bar{x}, \bar{v})$.
The following lemma is similar to Lemma 4.3 in \cite{Bre22} and will be used in the comparison argument later:
\begin{lemma}\label{Lap_lowb_submfd}
$n-\Delta^{\ac}_\Sigma\phi(\bar{x})-\langle H(\bar{x}), \bar{v}\rangle\geq 0$.
\end{lemma}
\begin{proof}
From Corollary \ref{superdiff_normalfib}, we have $\bar{\gamma}(1)=\Phi(\bar{x}, \bar{v})\in \pt^{c_+}\phi(\bar{x})$.
Thus the function $\frac{1}{2}d^2_{\bar{\gamma}(1)}|_\Sigma-\phi$
on $\Sigma$ is minimized at $\bar{x}$.
For any curve smooth $\gamma:[0,1]\to M$ satisfying $\gamma(0)\in \Sigma$ and $\gamma(1)=\bar{\gamma}(1)$, we then have
\begin{equation*}
    \begin{split}
    \frac{1}{2}\int_0^1 |\gamma'(t)|^2dt-\phi(\gamma(0))
    &\geq \frac{1}{2}\left(\int_0^1 |\gamma'(t)|dt\right)^2-\phi(\gamma(0))\\
    &\geq \frac{1}{2}d^2(\gamma(0), {\gamma}(1))-\phi(\gamma(0))\\
    &\geq \frac{1}{2}d^2(\bar{\gamma}(0), \bar{\gamma}(1))-\phi(\bar{\gamma}(0))\\
    &=\frac{1}{2}\int_0^1 |\bar{\gamma}'(t)|^2dt-\phi(\bar{\gamma}(0)).        
    \end{split}
\end{equation*}
That is, $\bar{\gamma}$ minimizes the the functional $\frac{1}{2}\int_0^1 |\gamma'(t)|^2dt-\phi(\gamma(0))$ among such curves $\gamma$.
Hence by the second variation formula, we get
\begin{equation*}
    \begin{split}
    -\langle \sff (Z(0), Z(0)), \bar{\gamma}'(0)\rangle
    &-\Hess^\Sigma_{\bar{x}}\phi(Z(0), Z(0))
    \\
    &+\int_0^1\left(|Z'(t)|^2-R(\bar{\gamma}'(t), Z(t), \bar{\gamma}'(t), Z(t))\right)dt\geq 0        
    \end{split}
\end{equation*}
    for any smooth vector field $Z$ along $\bar{\gamma}$ satisfying $Z(0)\in T_{\bar{x}}\Sigma$ and $Z(1)=0$.
\par
Now let $e_1, \ldots, e_n$ be an orthonormal basis of $T_{\bar{x}}\Sigma$, and let $E_i$ be the parallel transport of $e_i$ along the geodesic $\bar{\gamma}$.
Plugging in $\bar{\gamma}'(0)=-\nabla^\Sigma \phi(\bar{x})+\bar{v}$, $Z(t)=(1-t)E_i(t)$ into the formula above and summing over $i=1, \ldots, n$, we get
\begin{equation*}
    -\langle H(\bar{x}), \bar{v}\rangle
    -\Delta^{\ac}_\Sigma \phi(\bar{x})+n
    \geq \int_0^1 (1-t)^2  \Ric_n(\cal{P}_t, \bar{\gamma}'(t))dt\geq 0,    
\end{equation*}
where $\cal{P}_t$ is the $n$-plane in $T_{\bar{\gamma}(t)}M$ spanned by $\{E_i\}$, and we use the condition $\Ric_n \geq 0$.
\end{proof}

The following computation is the key of the whole estimation on $\det D\Phi(\bar{x}, \bar{v})$.
It can be viewed as a slight improvement of Proposition 4.6 in \cite{Bre22}, and the idea is to consider the traces of two diagonal blocks instead of the individual diagonal entries.
\begin{proposition}\label{Jacobian_est_along_geodesic_submfd}
The function
\begin{equation*}
    t^{-m}\left(1-\frac{t}{n}(\Delta^{\ac}_\Sigma \phi(\bar{x})+\langle H(\bar{x}), \bar{v}\rangle)\right)^{-n}|\det D\Phi_t(\bar{x}, \bar{v})|    
\end{equation*}
is decreasing in $t\in(0,1)$, where $\Phi_t$ is defined in Line \ref{def_geo_var}.
\end{proposition}
\begin{proof}
Let $e_1, \ldots, e_n$ be an orthonormal basis of $T_{\bar{x}}\Sigma$.
Let $(x_1,\hdots,x_n)$ be a system of geodesic normal coordinates on $\Sigma$ around the point $\bar{x}$ with $\frac{\partial}{\partial x_i} = e_i$ at $\bar{x}$. 
Let $\{\nu_{n+1},\hdots,\nu_{n+m}\}$ be a local orthonormal frame of $T^\perp \Sigma$, chosen so that $\langle \nabla_{e_i} \nu_\alpha,\nu_\beta \rangle = 0$ at $\bar{x}$. 
We write a normal vector $v$ as $v = \sum_{\alpha=n+1}^{n+m} v_\alpha \nu_\alpha$. 
With this understood, $(x_1,\hdots,x_n,v_{n+1},\hdots,v_{n+m})$ is a local coordinate system on the total space of $T^\perp \Sigma$.
\par
For each $1 \leq i \leq n$, we denote by $E_i(t)$ the parallel transport of $e_i$ along $\bar{\gamma}$.
Moreover, for each $1 \leq i \leq n$, we denote by $X_i(t)=(D \Phi_t)e_i$ the unique Jacobi field along $\bar{\gamma}$ satisfying $X_i(0) = e_i$ and 
\begin{equation*}
    \begin{cases}
    \langle X_i'(0),e_j \rangle &= -(\Hess^\Sigma \phi)(e_i,e_j) - \langle \sff(e_i,e_j),\bar{v} \rangle, \\ 
    \langle X_i'(0),\nu_\beta \rangle &= -\langle \sff(e_i,\nabla^\Sigma \phi(\bar{x})),\nu_\beta \rangle
    \end{cases}
\end{equation*}
for all $1 \leq j \leq n$ and all $n+1 \leq \beta \leq n+m$.
For each $n+1 \leq \alpha \leq n+m$, we denote by $N_\alpha(t)$ the parallel transport of $\nu_\alpha$ along $\bar{\gamma}$.
Moreover, for each $n+1 \leq \alpha \leq n+m$, we denote by $Y_\alpha(t)=(D\Phi_t) \nu_\alpha$ the unique Jacobi field along $\bar{\gamma}$ satisfying $Y_\alpha(0) = 0$ and $Y_\alpha'(0) = \nu_\alpha$.
We have that $X_1(t),\hdots,X_n(t),Y_{n+1},\hdots,Y_{n+m}(t)$ are linearly independent for each $t \in (0,1)$ since  $\bar{\gamma}$ is minimizing from Lemma \ref{Tangency_submfd},
\par
Let us define an $(n+m) \times (n+m)$-matrix $P(t)$ by 
\begin{equation*}
    \begin{array}{ll}
P_{ij}(t) = \langle X_i(t),E_j(t) \rangle, & P_{i\beta}(t) = \langle X_i(t),N_\beta(t) \rangle, \\ 
P_{\alpha j}(t) = \langle Y_\alpha(t),E_i(t) \rangle, & P_{\alpha\beta}(t) = \langle Y_\alpha(t),N_\beta(t) \rangle ,
\end{array}    
\end{equation*}
for $1 \leq i,j \leq n$ and $n+1 \leq \alpha,\beta \leq n+m$. Moreover, we define an $(n+m) \times (n+m)$-matrix $S(t)$ by 
\begin{equation*}
    \begin{array}{ll}
S_{ij}(t) = R(\bar{\gamma}'(t),E_i(t),\bar{\gamma}'(t),E_j(t)),
& S_{i\beta}(t) = R(\bar{\gamma}'(t),E_i(t),\bar{\gamma}'(t),N_\beta(t)), \\ 
S_{\alpha j}(t) = R(\bar{\gamma}'(t),N_\alpha(t),\bar{\gamma}'(t),E_j(t)),
& S_{\alpha\beta}(t) = R(\bar{\gamma}'(t),N_\alpha(t),\bar{\gamma}'(t),N_\beta(t)) ,
\end{array}    
\end{equation*}
for $1 \leq i,j \leq n$ and $n+1 \leq \alpha,\beta \leq n+m$. Clearly  $S(t)$ is symmetric.
Since the vector fields $X_1(t),\hdots,X_n(t),Y_{n+1}(t),\hdots,Y_{n+m}(t)$ are Jacobi fields, we obtain
\begin{equation*}
    P''(t) = -P(t) S(t).    
\end{equation*}
The initial conditions of $P$ are
\begin{equation*}
P(0) = 
\begin{bmatrix} \delta_{ij}  & 0 \\ 0 & 0 \end{bmatrix},\;
P'(0) = 
\begin{bmatrix} (-\Hess^\Sigma \phi)(e_i,e_j) - \langle \sff(e_i,e_j),\bar{v} \rangle & \langle -\sff(e_i,\nabla^\Sigma \phi),\nu_\beta \rangle \\ 0 & \delta_{\alpha\beta} \end{bmatrix}.
\end{equation*}
In particular, the matrix $P'(0) P(0)^T$ is symmetric. Moreover, the matrix 
\begin{equation*}
   \frac{d}{dt} (P'(t) P(t)^T) = P''(t) P(t)^T + P'(t) P'(t)^T = -P(t) S(t) P(t)^T + P'(t) P'(t)^T    
\end{equation*}
is symmetric for each $t$. Thus, we conclude that the matrix $P'(t) P(t)^T$ is symmetric for all $t$. 
\par

Since $X_1(t),\hdots,X_n(t),Y_{n+1},\hdots,Y_{n+m}(t)$ are linearly independent for each $t \in (0,1)$, the matrix $P(t)$ is invertible for $t \in (0,1)$. 
With $P'(t) P(t)^T$ being symmetric, it follows that the matrix $Q(t) := P(t)^{-1} P'(t)$ is symmetric. 
Furthermore, $Q$ satisfies the Riccati equation 
\begin{equation}\label{Riccati_eq_submfd}
    Q'(t) = P(t)^{-1} P''(t) - P(t)^{-1} P'(t) P(t)^{-1} P'(t) = -S(t) - Q(t)^2    
\end{equation}
for all $t \in (0,1)$.
We write
\begin{equation*}
    Q=\begin{bmatrix}
    Q_1&Q_2\\Q_2^t&Q_3
    \end{bmatrix}    
\end{equation*}
where $Q_1$ is $n\times n$, $Q_2$ is $n\times m$, and $Q_3$ is $m\times m$, and $Q_1, Q_3$ are symmetric.
\par
We now use the curvature conditions to derive differential inequalities for $\tr Q_1$ and $\tr Q_3$.
Let $\tr_n$ and $\tr_m$ denote the trace of the upper-left  $n\times n$ block and the lower-right $m\times m$ block in an $(n+m)\times (n+m)$ matrix, respectively.
First of all, since $\Ric_n, \Ric_m\geq 0$, we have
\begin{align*}
    -\tr_n S(t)&=-\Ric_n(\cal{E}_t, \bar{\gamma}'(t)))\leq 0;\\
    -\tr_m S(t)&=-\Ric_m(\cal{N}_t, \bar{\gamma}'(t))\leq 0,      
\end{align*}
where $\cal{E}_t, \cal{N}_t \subset T_{\bar{\gamma}(t)}M$ are the the spaces spanned by $\{E_i\}_{i=1}^n$ and $ \{N_\alpha\}_{\alpha=n+1}^{n+m}$, respectively.
Next, by Cauchy inequality we have
\begin{align*}
    \tr_n (Q^2)&=\tr(Q_1^2+Q_2Q_2^t)\geq \tr (Q_1^2)\geq \frac{1}{n}(\tr Q_1)^2;\\
    \tr_m (Q^2)&=\tr(Q_2^tQ_2+Q_3^2)\geq \tr(Q_3^2)\geq \frac{1}{m}(\tr Q_3)^2.      
\end{align*}
Combined with the Riccati equation \ref{Riccati_eq_submfd} above, we have the differential inequalities:
\begin{equation*}
    \tr Q_1'\leq -\frac{1}{n} (\tr Q_1)^2,\;
    \tr Q_3'\leq -\frac{1}{m} (\tr Q_3)^2. 
\end{equation*}
To derive the initial conditions of $\tr Q_1$ and $\tr Q_3$, consider the asymptotic expansion
\begin{equation}\label{asymp_P}
    P(t) = \begin{bmatrix} \delta_{ij} + O(t) & O(t) \\ O(t) & t \, \delta_{\alpha\beta} + O(t^2) \end{bmatrix},    
\end{equation}
implying
\begin{equation*}
    P(t)^{-1} = \begin{bmatrix} \delta_{ij} + O(t) & O(1) \\ O(1) & t^{-1}  \delta_{\alpha\beta} + O(1) \end{bmatrix} 
\end{equation*}
as $t \to 0$. 
Moreover, using the Jacobi field equations we have
\begin{equation*}
    P'(t) = \begin{bmatrix} (-\Hess^\Sigma \phi)(e_i,e_j) - \langle \sff(e_i,e_j),\bar{v} \rangle + O(t) & O(1) \\ O(t) & \delta_{\alpha\beta} + O(t) \end{bmatrix}    
\end{equation*}
as $t \to 0$. Consequently, the matrix $Q(t) = P(t)^{-1} P'(t)$ satisfies the asymptotic expansion 
\begin{equation*}
    Q(t) = \begin{bmatrix} (-\Hess^\Sigma \phi)(e_i,e_j) - \langle \sff(e_i,e_j),\bar{v} \rangle + O(t) & O(1) \\ O(1) & t^{-1} \delta_{\alpha\beta} + O(1) \end{bmatrix}    
\end{equation*}
as $t \to 0$, and we thus have the initial conditions
\begin{equation}\label{init_Q_submfd}
    \tr Q_1\sim -\Delta^{\ac}_\Sigma \phi(\bar{x}) - \langle H(\bar{x}),\bar{v}\rangle+O(t),\;    \tr Q_3 \sim \frac{m}{t}+O(1)
\end{equation}
as $t \to 0$.
Hence we can find a small number $\tau_0 \in (0,1)$ such that
\begin{equation*}
    \tr Q_1(\tau)<-\Delta^{\ac}_\Sigma \phi(\bar{x}) - \langle H(\bar{x}),\bar{v}\rangle+\sqrt{\tau}=:G(\tau),\;
    0<\tr Q_3(\tau)<\frac{2m}{\tau}    
\end{equation*}
for all $\tau \in (0,\tau_0)$. 
For such a fixed $\tau$, a standard ODE comparison principle together with Lemma \ref{Lap_lowb_submfd} implies
\begin{equation*}
    \tr Q_1(t)\leq \frac{G(\tau)}{1+\frac{(t-\tau)G(\tau)}{n}},\;
    \tr Q_3(t)\leq\frac{m}{t-\frac{\tau}{2}}   
\end{equation*}
for all $t \in (\tau,1)$.
Passing to the limit as $\tau \to 0$, we conclude that 
\begin{equation*}
        \tr Q_1(t)\leq \frac{-\Delta^{\ac}_\Sigma \phi(\bar{x}) - \langle H(\bar{x}),\bar{v}\rangle}{1-\frac{t}{n}(\Delta^{\ac}_\Sigma \phi(\bar{x}) + \langle H(\bar{x}),\bar{v}\rangle)},\;
    \tr Q_3(t)\leq\frac{m}{t}
\end{equation*}
for $t\in (0,1)$, and therefore
\begin{equation}\label{est_trQ}
\tr Q(t)=\tr Q_1(t)+\tr Q_3(t)\leq \frac{m}{t}+ \frac{-\Delta^{\ac}_\Sigma \phi(\bar{x}) - \langle H(\bar{x}),\bar{v} \rangle }{1+\frac{t}{n}(-\Delta^{\ac}_\Sigma \phi(\bar{x}) - \langle H(\bar{x}),\bar{v} \rangle)}        
\end{equation}
for all $t \in (0,1)$.
\par
We can now estimate $\det P$.
From the asymptotic expansion of $P$ in Line (\ref{asymp_P}), we have $\lim_{t \to 0} t^{-m} \det P(t) = 1$. 
Since $P(t)$ is invertible for each $t \in (0,1)$, it follows that $\det P(t) > 0$ for such $t$, and thus we can consider the function $\log \det P(t)$.
Using the estimate of $\tr Q$ in Line \ref{est_trQ} above, we obtain 
\begin{align*}
    \frac{d}{dt} \log \det P(t) = \text{\rm tr}(Q(t)) 
&\leq \frac{m}{t}+ \frac{-\Delta^{\ac}_\Sigma \phi(\bar{x}) - \langle H(\bar{x}),\bar{v} \rangle }{1+\frac{t}{n}(-\Delta^{\ac}_\Sigma \phi(\bar{x}) - \langle H(\bar{x}),\bar{v} \rangle)}\\
&=m\frac{d}{dt}\log(t) + n\frac{d}{dt} \log \left(1+\frac{t}{n}(-\Delta^{\ac}_\Sigma \phi(\bar{x}) - \langle H(\bar{x}),\bar{v} \rangle)\right)
\end{align*}
for all $t \in (0,1)$. Consequently, the function
\begin{equation*}
    t \mapsto t^{-m} \left(1+\frac{t}{n}(-\Delta^{\ac}_\Sigma \phi(\bar{x}) - \langle H(\bar{x}),\bar{v} \rangle)\right)^{-n}\det P(t) 
\end{equation*}
is  decreasing for $t \in (0,1)$. 
Finally observe that $|\det D\Phi_t( \bar{x}, \bar{v})| =
\det P(t)$ for all $t \in(0,1)$ since $\{E_i\}_{i=1}^n\cup  \{N_\alpha\}_{\alpha=n+1}^{n+m}$ forms an orthonormal basis.
This finishes the proof of the proposition.
\end{proof}

Since $\lim_{t \to 0} t^{-m}|\det D\Phi_t(\bar{x},\bar{v})| = 1$ and $\Phi_1=\Phi$, we have the estimate:
\begin{corollary}[\cite{Bre22}, Corollary 4.7]\label{Jacobian_est_submfd}
The Jacobian determinant of $\Phi$ satisfies
\begin{equation*}
    |\det D\Phi(x,v)| \leq \left(1-\frac{1}{n}(\Delta^{\ac}_\Sigma \phi(x) + \langle H(x),v \rangle) \right)^n    
\end{equation*}
for $(x,v) \in A$. 
\end{corollary}

The final step is to combine the results in the previous steps.
From the change of variable formula (Line \ref{Jacobian_det_Phi}) and the corollary above, for $\mu$-a.e.\@ $x\in \Sigma$ we have
\begin{align*}
    \frac{\vol(\Omega)}{\int_\Sigma f^\frac{n}{n-1}} f^\frac{n}{n-1}(x)
    = \int_{A_x}|\det D\Phi(x,v)|dv
    \leq \int_{A_x}\left(1-\frac{1}{n}(\Delta^{\ac}_\Sigma \phi(x) + \langle H(x),v \rangle) \right)^n dv.
\end{align*}
    Since $\Phi(x,v)\in \Omega$, we have $|v|\leq d(x, \Phi(x,v))\leq r$.
    Thus by Cauchy inequality, we have
\begin{equation*}
    \int_{A_x}\left(1-\frac{1}{n}(\Delta^{\ac}_\Sigma \phi(x) + \langle H(x),v \rangle) \right)^n dv
    \leq \int_{A_x}\left( 1-\frac{1}{n}\Delta^{\ac}_\Sigma \phi(x)+\frac{1}{n}|H(x)|r\right)^n dv,    
\end{equation*}
and the integrand now no longer depends on $v$.
Again since $\Phi(x,v)\in \Omega$, we have
\begin{equation}\label{est_along_normal_fiber}
    \begin{split}
            \int_{A_x}dv
    &\leq \int_{\{v\in T^\perp_x\Sigma:\sigma^2r^2<|v|^2+|\nabla \phi(x)|^2<r^2}dv\\
    &= |B^m|\left( (r^2-|\nabla \phi(x)|^2)^\frac{m}{2}-(\sigma^2r^2-|\nabla \phi(x)|^2)_+^\frac{m}{2}\right)\\
    &\leq  \frac{m}{2}|B^m|r^m(1-\sigma^2),
    \end{split}
\end{equation}
where we use the mean value theorem with $m\geq 2$ in the last step.
Consequently, we have the following inequality:
\begin{equation*}
    \frac{\vol(\Omega)}{\int_\Sigma f^\frac{n}{n-1}} f^\frac{n}{n-1}(x)\leq \left( 1-\frac{1}{n}\Delta^{\ac}_\Sigma \phi(x)+\frac{1}{n}|H(x)|r\right)^n \frac{m}{2}|B^m|r^m(1-\sigma^2).
\end{equation*}
Taking the $n$-th root, multiplying by $f$, and integrating over $\Sigma$  on both sides, we get
\begin{equation*}
    n\left(\frac{2\vol(\Omega)}{m|B^m|(1-\sigma^2)}\right)^\frac{1}{n}\left( \int_\Sigma f^\frac{n}{n-1}\right)^\frac{n-1}{n} 
    \leq r^\frac{m}{n}\int_\Sigma (nf-f\Delta^{\ac}_\Sigma \phi + rf|H|).    
\end{equation*}
Using the inequality from integration by parts (Lemma \ref{int_by_part_submfd}), we get 
\begin{equation}\label{submfd_intRic_nonneg_before_limit}
   n\left(\frac{2\vol(\Omega)}{m|B^m|(1-\sigma^2)}\right)^\frac{1}{n}\left( \int_\Sigma f^\frac{n}{n-1}\right)^\frac{n-1}{n} 
   \leq r^\frac{m}{n}\left(n\int_\Sigma f+
   r\int_{\pt \Sigma}f+r\int_\Sigma |\nabla^\Sigma f |+r\int_\Sigma f|H|\right).
\end{equation}
Recall the definition of the set $\Omega=\{y\in M \mid \sigma r\leq d(y,x)\leq r, \forall x\in \Sigma\}$.
Since $\sigma<1$, we have $\sigma (2-\delta)<\delta$ for any $\delta<1$ close enough to $1$.
Choose any $x_0\in \Sigma$.
By triangle inequality, we then have $B_{\sigma (2-\delta)r, \delta r}(x_0)\subset \Omega$
for $r$ large enough, where $B_{a,b}(x_0)$ denotes the annulus in $M$ centered at $x_0$ with radii $a<b$.
We divide both sides of Line \ref{submfd_intRic_nonneg_before_limit} by $r^\frac{n+m}{n}$ and let $r \to \infty$ while keeping $\sigma, \delta$ fixed. 
This gives 
\begin{equation*}
    n\theta^\frac{1}{n}\left(\frac{2|B^{n+m}|(\delta^{n+m}-(\sigma(2-\delta))^{n+m})}{m|B^m|(1-\sigma^2)}\right)^\frac{1}{n}\left( \int_\Sigma f^\frac{n}{n-1}\right)^\frac{n-1}{n} 
   \leq \int_{\pt \Sigma}f+\int_\Sigma |\nabla^\Sigma f |+\int_\Sigma f|H|.
\end{equation*}
Finally we take the limit $\delta\to 1^-$ and then let $\sigma \to 1^-$ to conclude
\begin{equation*}
       n\theta^\frac{1}{n}\left(\frac{(n+m)|B^{n+m}|}{m|B^m|}\right)^\frac{1}{n}\left( \int_\Sigma f^\frac{n}{n-1}\right)^\frac{n-1}{n} 
   \leq\int_{\pt \Sigma}f+\int_\Sigma |\nabla^\Sigma f |+\int_\Sigma f|H|.
\end{equation*}
This finishes the proof of Theorem Theorem \ref{Intro_Mic_Sim_interRicci_nonneg}.

\subsection{Manifolds with positive lower bounds on intermediate Ricci curvatures}\label{Subsec_pos_inter_Ricci}
In this subsection, we assume $\Ric_n\geq (n-1)k_1, \Ric_m\geq (m-1)k_2$ for some $k_1, k_2>0$ on the ambient complete manifold $M$.
Then since $\Ric\geq (n-1)k_1+(m-1)k_2>0$, Myers' theorem (see e.g. \cite{Pet16}, Theorem 6.3.3) yields that $M$ is compact (without boundary).
Hence there is no guarantee that the torus-like subset $
    \Omega=\{p\in M: \sigma r\leq d(x,p)\leq r, \forall x\in \Sigma\}
$
considered in the previous subsection will have positive volume since $r$ cannot be arbitrarily large.
Nevertheless, taking $\Omega=M$, $r=\diam(M)$ in the proof of Theorem \ref{Intro_Mic_Sim_interRicci_nonneg} (see Subsection \ref{Subsec_nonneg_inter_Ricci} above) with little adjustment in Line \ref{est_along_normal_fiber} yields Corollary \ref{Intro_Mic_Sim_M_closed_int_Ric_nonneg}.
\par
On the other hand, it is possible to derive a more general result, namely Theorem \ref{Intro_interRicci_pos_tubular}, by modifying the proof of Theorem \ref{Intro_Mic_Sim_interRicci_nonneg}.
The idea is to use the positive bounds on the intermediate Ricci curvatures in the  Riccati equation (Line \ref{Riccati_eq_submfd}) to get another type of differential inequalities.
In that case, the speed of the transporting geodesic $|\gamma'(t)|=|\nabla^\Sigma \phi(x)+v|$ comes into play, which can be controlled by our choice of the target domain $\Omega$ since $\gamma$ is a minimizing.
We here choose $\Omega$ as the complement of a tubular neighborhood of $\Sigma$ to obtain a lower bound on the speed.
\begin{proof}[Proof of Theorem \ref{Intro_interRicci_pos_tubular}]
Most of the notations and arguments are from the previous Subsection \ref{Subsec_nonneg_inter_Ricci}. Recall that $N_\epsilon$ is the $\epsilon$-tubular neighborhood of $\Sigma$ for some $\epsilon>0$, and we may assume that $\Omega=M\backslash N_\epsilon$ has positive volume.
Consider the optimal transport problem between the following two probability measures:
\begin{equation*}
    \mu=\frac{1}{\int_\Sigma f^\frac{n}{n-1}}f^\frac{n}{n-1} \vol_\Sigma,\;
    \nu=\frac{1}{\vol(\Omega)}\vol_M\llcorner \Omega
\end{equation*}
on $\Sigma$ and $\Omega$ respectively, with the same cost function $c(x,\zeta)=\frac{1}{2}d_M^2(x,\zeta)$ on $\Sigma \times \Omega$.
By the generalized McCann's theorem (Theorem \ref{gen_McCann_thm}), we have a $c$-concave function $\phi$ on $\Sigma$, and the map 
\begin{equation*}
    \Phi(x,v):=\exp_x(-\nabla^\Sigma \phi(x)+v)    
\end{equation*}
defined on the set $A\subset T^\perp \Sigma$ such that $\Phi(A)$ has full measure in $\Omega$.
Since $\Phi(x,v)\in \Omega$, we have
\begin{equation*}
    \epsilon^2\leq d^2(x, \Phi(x,v))= |\nabla^\Sigma (x)|^2+|v|^2\leq \diam(M)^2.    
\end{equation*}
Now we estimate $|\det D\Phi(x,v)|$ for fixed $(x,v)\in A$ by modifying Proposition \ref{Jacobian_est_along_geodesic_submfd}.
Recall the Riccati equation (Line \ref{Riccati_eq_submfd}):
\begin{equation*}
    Q'=-S-Q^2.    
\end{equation*}
By $\Ric_n\geq (n-1)k_1, \Ric_m \geq (m-1)k_2$ and Cauchy inequality, we get
\begin{equation*}
        \tr Q_1'\leq -(n-1)k_1 \epsilon^2 -\frac{1}{n} (\tr Q_1)^2,\;
    \tr Q_3'\leq -(m-1)k_2\epsilon^2 -\frac{1}{m} (\tr Q_3)^2.
\end{equation*}
We have the same initial conditions for $\tr Q_1$ and $\tr Q_3$ as in Line \ref{init_Q_submfd}.
Since $k_1, k_2>0$ and $n,m\geq 2$, standard ODE comparison arguments as those from Line \ref{init_Q_submfd} to Line \ref{est_trQ} then imply
\begin{equation*}
    \frac{d}{dt}\log \det P
    =\tr Q=\tr Q_1(t)+\tr Q_3(t)
    \leq \frac{d}{dt}(n\log \cos G_1+m\log \cos G_2)    
\end{equation*}
for $t\in (0,1)$, where 
\begin{equation*}
        G_1(t)=-t\epsilon\sqrt{\frac{k_1(n-1)}{n}}+\arctan\left(\frac{-\Delta^{\ac}_\Sigma \phi(x)-\langle H(x), v\rangle}{\epsilon\sqrt{k_1n(n-1)}}\right);\;
    G_2(t)=-t\epsilon\sqrt{\frac{k_2(m-1)}{m}}+\frac{\pi}{2}.
\end{equation*}
\begin{remark}\label{rmk_nonblowup}
To see that there is no blow-up for $t\in (0,1)$, assume the contrary that there is some $0<t_0<1$ such that $G_1(t)$ tends to $\frac{-\pi}{2}$ as $t\to t_0^-$.
Then $\tr Q$ will tend to $-\infty$ as $t\to t_0^-$.
However this is impossible since $\tr Q$ is continuous on $[0,1)$.
Thus the inequality will hold for $t\in (0,1)$ and in fact $G_1(t)\in (\frac{-\pi}{2}, \frac{\pi}{2})$.
For the same reason we have $G_2(t)\in (\frac{-\pi}{2}, \frac{\pi}{2})$.
\end{remark}
Consequently, the following function in $t$ is decreasing:
\begin{equation*}
     \frac{\det P(t)}{\cos^n G_1(t) \cos^m G_2(t)}.   
\end{equation*}
Since $\lim_{t \to 0} t^{-m}  \det P(t) = 1$ and $\det P(1)=|\det D\Phi(x,v)|$, we have
\begin{align*}
    &|\det D\Phi(x,v)|\leq \lim_{t\to 0}\frac{\cos^n G_1(1)\cos^mG_2(1)\det P(t)}{\cos^n G_1(t)\cos^m G_2(t)}\\
    \begin{split}
    &=\left[\cos\left(\epsilon\sqrt{\frac{k_1(n-1)}{n}}\right)
    -\frac{1}{n}\sinc \left(\epsilon\sqrt{\frac{k_1(n-1)}{n}}\right)({\Delta^{\ac}_\Sigma  \phi(x)+\langle H(x), v\rangle})\right]^n\\
    &\quad\cdot\sinc^m\left(\epsilon\sqrt{\frac{k_2(m-1)}{m}}\right)
    \end{split}\\
    \begin{split}
    &\leq \left[\cos\left(\epsilon\sqrt{\frac{k_1(n-1)}{n}}\right)
    -\frac{1}{n}\sinc \left(\epsilon\sqrt{\frac{k_1(n-1)}{n}}\right)({\Delta^{\ac}_\Sigma  \phi(x)-| H(x)|\diam(M)})\right]^n\\
    &\quad \cdot\sinc^m\left(\epsilon\sqrt{\frac{k_2(m-1)}{m}}\right),
    \end{split}
\end{align*}
where we use that $|v|\leq \diam(M)$.
Combined with the change of variable formula of $D\Phi$ (Line \ref{Jacobian_det_Phi}) and using $A_x\subset \{v\in T^{\perp}_x \Sigma \mid |v|\leq \diam(M)\}$, we have
\begin{align*}
    &\frac{\vol(\Omega)}{\int_\Sigma f^\frac{n}{n-1}} f^\frac{n}{n-1}(x)= \int_{A_x}|\det D\Phi(x,v)|dv\\
    \begin{split}
            &\leq \left[\cos\left(\epsilon\sqrt{\frac{k_1(n-1)}{n}}\right)
    -\frac{1}{n}\sinc \left(\epsilon\sqrt{\frac{k_1(n-1)}{n}}\right)({\Delta^{\ac}_\Sigma  \phi(x)-| H(x)|\diam(M)})\right]^n\\
    &\quad \cdot 
    \sinc^m\left(\epsilon\sqrt{\frac{k_2(m-1)}{m}}\right)|B^m|\diam(M)^m.
    \end{split}
\end{align*}
Now we take the $n$-th root on both sides, multiply by $f$, integrate over $\Sigma$ and use the integration by parts inequality (Lemma \ref{int_by_part_submfd}) with $r=\diam(M)$ to finish the proof.
\end{proof}

\subsection{Manifolds with negative lower bounds on intermediate Ricci curvatures}

We prove Theorem \ref{Intro_interRicci_neg_local} in this subsection, where we assume $\Ric_n\geq nk_1, \Ric_m\geq mk_2$ for some $k_1, k_2<0$ on the ambient complete manifold $M$.
As in Subsection \ref{Subsec_pos_inter_Ricci}, these lower bounds will be used in the Riccati equation (Line \ref{Riccati_eq_submfd}) to give another type of differential inequalities involving the speed of transporting geodesics.
Here we choose the target domain $\Omega$ as a geodesic ball to get an upper bound on the speed.

\begin{proof}[Proof of Theorem \ref{Intro_interRicci_neg_local}]
Again, most of the notations and arguments are from Subsection \ref{Subsec_nonneg_inter_Ricci}.
Consider the optimal transport problem between the following two probability measures:
\begin{equation*}
    \mu=\frac{1}{\int_\Sigma f^\frac{n}{n-1}}f^\frac{n}{n-1} \vol_\Sigma,\;
    \nu=\frac{1}{\vol(\Omega)}\vol_M\llcorner \Omega
\end{equation*}
on $\Sigma$ and $\Omega=B^M_{\frac{r}{2}}(x_0)$, respectively, with the cost function $c(x,\zeta)=\frac{1}{2}d_M^2(x,\zeta)$ on $\Sigma \times \Omega$.
By the generalized McCann's theorem (Theorem \ref{gen_McCann_thm}), we have a $c$-concave function $\phi$ on $\Sigma$, and the map 
\begin{equation*}
    \Phi(x,v):=\exp_x(-\nabla^\Sigma \phi(x)+v)    
\end{equation*}
defined on $A\subset T^\perp \Sigma$ such that $\Phi(A)$ has full measure in $\Omega$.
Since $\Phi(x,v)\in \Omega$, we have
\begin{equation*}
    |\nabla^\Sigma (x)|^2+|v|^2=d^2(x, \Phi(x,v))\leq r^2.    
\end{equation*}
Now we estimate $|\det D\Phi(x,v)|$ for fixed $(x,v)\in A$ by modifying Proposition \ref{Jacobian_est_along_geodesic_submfd}.
Recall the Riccati equation (Line \ref{Riccati_eq_submfd}):
\begin{equation*}
   Q'=-S-Q^2.    
\end{equation*}
By $\Ric_n\geq nk_1, \Ric_m \geq mk_2$ and Cauchy inequality, we get
\begin{equation*}
    \tr Q_1'\leq -nk_1 r^2 -\frac{1}{n} (\tr Q_1)^2,\;
    \tr Q_3'\leq -mk_2r^2 -\frac{1}{m} (\tr Q_3)^2.    
\end{equation*}
Again, we have the same initial conditions for $\tr Q_1$ and $\tr Q_3$ as in Line \ref{init_Q_submfd}, and hence we can find a small number $\tau_0 \in (0,1)$ such that
\begin{equation*}
    \tr Q_1(\tau)<-\Delta^{\ac}_\Sigma \phi(\bar{x}) - \langle H(\bar{x}),\bar{v}\rangle+\sqrt{\tau},\;
    \tr Q_3(\tau)<\frac{2m}{\tau},\;
    -mk_2r^2 -\frac{1}{m} (\tr Q_3(\tau))^2<0    
\end{equation*}
for all $\tau \in (0,\tau_0)$. 
By a standard ODE comparison principle and taking the limit as $\tau\to 0^+$, we conclude
\begin{equation*}
    \frac{d}{dt}\log \det P
    =\tr Q_1(t)+\tr Q_3(t)
    \leq \frac{d}{dt}(n\log \cosh G_1+m\log \sinh G_2),    
\end{equation*}
for $t\in (0,1)$, where 
\begin{equation*}
    G_1(t)=tr\sqrt{-k_1}+\tanh^{-1}\left(\frac{-\Delta^{\ac}_\Sigma \phi(x)-\langle H(x), v\rangle}{rn\sqrt{-k_1}}\right),\,
    G_2(t)=tr\sqrt{-k_2}.    
\end{equation*}
Therefore the following function in $t$ is decreasing:
\begin{equation*}
    \frac{\det P(t)}{\cosh^n G_1(t) \sinh^m G_2(t)}.    
\end{equation*}
Since $\lim_{t \to 0} t^{-m} \, \det P(t) = 1$ and $\det P(1)=|\det D\Phi(x,v)|$, we have
\begin{align*}
    &|\det D\Phi(x,v)|
    \leq \lim_{t\to 0}\frac{\cosh^n G_1(1)\sinh^mG_2(1)\det P(t)}{\cosh^n G_1(t)\sinh^m G_2(t)}\\
    &=\left[\cosh\left(r\sqrt{-k_1}\right)
    -\frac{\sinh(r\sqrt{-k_1})}{rn\sqrt{-k_1}}({\Delta^{\ac}_\Sigma  \phi(x)+\langle H(x), v\rangle})\right]^n
    \left(\frac{\sinh(r\sqrt{-k_2})}{r\sqrt{-k_2}}\right)^m\\
    &\leq \left[\cosh\left(r\sqrt{-k_1}\right)
    -\frac{\sinh(r\sqrt{-k_1})}{rn\sqrt{-k_1}}({\Delta^{\ac}_\Sigma  \phi(x)-|H(x)|r})\right]^n
    \left(\frac{\sinh(r\sqrt{-k_2})}{r\sqrt{-k_2}}\right)^m,
\end{align*}
where we use that $|v|\leq r$.
Combined with the change of variable formula of $D\Phi$ (Line \ref{Jacobian_det_Phi}) and using $A_x\subset \{v\in T^{\perp}_x \Sigma \mid |v|\leq r\}$, we have
\begin{align*}
    &\frac{\vol\left(B_{\frac{r}{2}}(x_0)\right)}{\int_\Sigma f^\frac{n}{n-1}} f^\frac{n}{n-1}(x)
    = \int_{A_x}|\det D\Phi(x,v)|dv\\
    &\leq \left[\cosh\left(r\sqrt{-k_1}\right)
    -\frac{\sinh(r\sqrt{-k_1})}{rn\sqrt{-k_1}}({\Delta^{\ac}_\Sigma  \phi(x)-|H(x)|r})\right]^n 
    \left(\frac{\sinh(r\sqrt{-k_2})}{r\sqrt{-k_2}}\right)^m|B^m|r^m.
\end{align*}
Now we take the $n$-th root on both sides, multiply by $f$, integrate over $\Sigma$ and use the integration by parts inequality (Lemma \ref{int_by_part_submfd}) to finish the proof.
\end{proof}

\appendix
\section{\texorpdfstring{$c$}{c}-concave functions on submanifolds}

    In this appendix, we show that $c$-concave functions on submanifolds are semiconcave.
    This follows essentially from the Hessian upper bound (see \cite{CMS01}, Lemma 3.11) of such functions.
    Recall the notations that $c=\frac{1}{2}d_M^2$ is the cost function defined on $\Sigma\times \Omega$, where $\Sigma\subset M$ is a compact submanifold, and $\Omega\subset M$ is a compact (regular) domain.
    \par
First we prove the following adaption of Lemma 3.12 in \cite{CMS01} under the submanifold setting:
\begin{lemma}[Hessian upper bound for distance squared restricted to $\Sigma$]\label{Hes_bd_d^2_submfd}
Fix $x\in \Sigma$ and $\zeta\in M$.
Let $\gamma$ be a minimal geodesic in $M$ from $x$ to $\zeta$.
Suppose $k<0$ is a lower bound for the sectional curvature along $\gamma$.
Define $b(s)=s\coth s$.
Then for each $u\in T_x\Sigma$, we have
\begin{equation*}
    \limsup_{s\to 0}\frac{d^2_\zeta (\exp^\Sigma_{x}su)+d^2_\zeta(\exp^\Sigma_x -su)-2d^2_\zeta(x)}{s^2}\leq 4b\left(\frac{d_\zeta(x)}{2}\sqrt{-k}\right)+2d_\zeta(x)|\sff(u,u)|.    
\end{equation*}
Here, $\sff$ is the second fundamental form of $\Sigma$.
\end{lemma}
\begin{proof}
Let $l=d_\zeta(x)$.
Assume that $\gamma:[0,l]\to M$ is parametrized by arc length.
Let $\tilde{\zeta}=\gamma(\frac{l}{2})$ be the middle point of $\gamma$. 
By triangle inequality, we have
\begin{equation*}
    d_\zeta^2(\exp^\Sigma_x su)\leq (d(\zeta, \tilde{\zeta})+d(\tilde{\zeta}, \exp^\Sigma_x su))^2\leq  2\left(\frac{l}{2}\right)^2+2d^2_{\tilde{\zeta}}(\exp^\Sigma_x su),    
\end{equation*}
and a similar bound holds for $d^2_\zeta (\exp^\Sigma_x -su)$.
Hence we have
\begin{equation*}
    \frac{d^2_\zeta (\exp^\Sigma_{x}su)+d^2_\zeta(\exp^\Sigma_x -su)-2d^2_\zeta(x)}{s^2}
    \leq 2\frac{d^2_{\tilde{\zeta}}(\exp^\Sigma_x su)+d^2_{\tilde{\zeta}}(\exp^\Sigma_x -su)-2d^2_{\tilde{\zeta}}(x)}{s^2}.    
\end{equation*}
Since $\gamma$ is minimizing, $x$ is not in the cut locus of the middle point $\tilde{\zeta}$.
Thus as $s\to 0$ the limit exists on the right hand side, and we have
\begin{align*}
    \limsup_{s\to 0}    \frac{d^2_\zeta (\exp^\Sigma_{x}su)+d^2_\zeta(\exp^\Sigma_x -su)-2d^2_\zeta(x)}{s^2}
    &\leq 2 \Hess^\Sigma_x (d^2_{\tilde{\zeta}}\big|_\Sigma)(u,u)\\
    &=2\left(\Hess^M_x (d^2_{\tilde{\zeta}})(u,u)+\langle \nabla (d^2_{\tilde{\zeta}})(x), \sff(u,u)\rangle\right).
\end{align*}
Now we apply Lemma 3.12 in \cite{CMS01} to $\tilde{\zeta}$ and obtain
\begin{equation*}
    \Hess^M_x (d^2_{\tilde{\zeta}})(u,u)\leq 2b\left(\frac{l}{2}\sqrt{-k}\right).    
\end{equation*}
The lemma then follows from Cauchy inequality and that $\left|\nabla (d^2_{\tilde{\zeta}})(x)\right|=2d_{\tilde{\zeta}}(x)=l$.
\end{proof}

We now prove the main result in this appendix, which can be viewed as the submanifold version of Proposition 3.14 in \cite{CMS01}:
\begin{proposition}[$c$-concave functions on $\Sigma$ are semiconcave]\label{submfd_c_concave_semiconcave}
Let $\Sigma\subset M$ be a compact submanifold, and let $\Omega\subset M$ be a compact domain. 
Let $\phi$ be a $c$-concave function on $\Sigma$ (w.r.t the pair $(\Sigma, \Omega)$).
Then $\phi$ is semiconcave on $\Sigma$.
\end{proposition}
\begin{proof}
For $x\in \Sigma$, since the function $\frac{1}{2}d^2_x-\phi^{c_+}$ is continuous on the compact set $\Omega$, there is some $\zeta\in \Omega$ such that 
\begin{equation*}
    \frac{1}{2}d^2_x(\zeta)-\phi^{c_+}(\zeta)=\inf_{\eta\in \Omega}\frac{1}{2}d^2(x,\eta)-\phi^{c_+}(\eta)=\phi^{c_+c_+}\phi(x)=\phi(x).    
\end{equation*}
Hence $\zeta\in \pt^{c_+}\phi(x)$.
Then for any $u\in T_x\Sigma$ of small norm, we have
\begin{equation*}
    \phi(\exp^\Sigma_xu)\leq \phi(x)+\frac{1}{2}d^2_{\zeta}(\exp^\Sigma_x u)-\frac{1}{2}d^2_\zeta (x).    
\end{equation*}
Therefore by Lemma \ref{Hes_bd_d^2_submfd}, we have for $u\in T_x \Sigma$ of unit length,
\begin{align*}
    &\limsup_{s\to 0}\frac{\phi(\exp^\Sigma_xsu)+\phi(\exp^\Sigma_x -su)-2\phi(x)}{s^2}\\
    \leq&\frac{1}{2}\limsup_{s\to 0}\frac{d^2_\zeta (\exp^\Sigma_{x}su)+d^2_\zeta(\exp^\Sigma_x -su)-2d^2_\zeta(x)}{s^2}\\
    \leq& 2b\left(\frac{l}{2}\sqrt{-k}\right)+l|\sff(u,u)|,
\end{align*}
where the function $b$ is defined in Lemma \ref{Hes_bd_d^2_submfd} and $l=d(x,\zeta)$.
Now since both $\Sigma$ and $\Omega$ are compact, $l$ admits a universal upper bound, and we can choose a universal lower bound $k$ of the sectional curvature along any minimal geodesic starting in $\Sigma$ and ending in  $\Omega$.
In addition, the second fundamental form $\sff$ admits a universal upper bound since $\Sigma$ is compact.
Hence we have
\begin{equation*}
    \limsup_{s\to 0}\frac{\phi(\exp^\Sigma_xsu)+\phi(\exp^\Sigma_x -su)-2\phi(x)}{s^2}\leq C    
\end{equation*}
for some constant $C$ independent of $x\in \Sigma$ and $u\in T_x \Sigma$.
By Lemma 3.11 in \cite{CMS01}, we conclude that $\phi$ is semiconcave on $\Sigma$.
\end{proof}

\printbibliography

@article {HK78,
    AUTHOR = {Heintze, Ernst and Karcher, Hermann},
     TITLE = {A general comparison theorem with applications to volume
              estimates for submanifolds},
   JOURNAL = {Ann. Sci. \'{E}cole Norm. Sup. (4)},
  FJOURNAL = {Annales Scientifiques de l'\'{E}cole Normale Sup\'{e}rieure. Quatri\`eme
              S\'{e}rie},
    VOLUME = {11},
      YEAR = {1978},
    NUMBER = {4},
     PAGES = {451--470},
      ISSN = {0012-9593},
   MRCLASS = {53C40 (58E10)},
  MRNUMBER = {533065},
MRREVIEWER = {Hubert Gollek},
       URL = {http://www.numdam.org/item?id=ASENS_1978_4_11_4_451_0},
}

@misc {Bre22,
  doi = {10.48550/ARXIV.2009.13717},
  
  url = {https://arxiv.org/abs/2009.13717},
  
  author = {Brendle, Simon},
  
  title = {Sobolev inequalities in manifolds with nonnegative curvature},
  
  publisher = {arXiv},
  
  year = {2022},
  note={To appear in Comm. Pure. Appl. Math}
}

@article {Bre21,
    AUTHOR = {Brendle, Simon},
     TITLE = {The isoperimetric inequality for a minimal submanifold in
              {E}uclidean space},
   JOURNAL = {J. Amer. Math. Soc.},
  FJOURNAL = {Journal of the American Mathematical Society},
    VOLUME = {34},
      YEAR = {2021},
    NUMBER = {2},
     PAGES = {595--603},
      ISSN = {0894-0347},
   MRCLASS = {53A10 (53A07)},
  MRNUMBER = {4280868},
MRREVIEWER = {Jianquan Ge},
       DOI = {10.1090/jams/969},
       URL = {https://doi.org/10.1090/jams/969},
}

@misc{BE22,
  doi = {10.48550/ARXIV.2205.10284},
  
  url = {https://arxiv.org/abs/2205.10284},
  
  author = {Simon Brendle and Michael Eichmair},
  
  title = {Proof of the Michael-Simon-Sobolev inequality using optimal transport},
  
  publisher = {arXiv},
  
  year = {2022},
  note={preprint},
}

@incollection {AG13,
    AUTHOR = {Ambrosio, Luigi and Gigli, Nicola},
     TITLE = {A user's guide to optimal transport},
 BOOKTITLE = {Modelling and optimisation of flows on networks},
    SERIES = {Lecture Notes in Math.},
    VOLUME = {2062},
     PAGES = {1--155},
 PUBLISHER = {Springer, Heidelberg},
      YEAR = {2013},
   MRCLASS = {49Q20 (35R70 49-01)},
  MRNUMBER = {3050280},
MRREVIEWER = {Luca Granieri},
       DOI = {10.1007/978-3-642-32160-3\_1},
       URL = {https://doi.org/10.1007/978-3-642-32160-3_1},
}

@article {McC01,
    AUTHOR = {McCann, Robert J.},
     TITLE = {Polar factorization of maps on {R}iemannian manifolds},
   JOURNAL = {Geom. Funct. Anal.},
  FJOURNAL = {Geometric and Functional Analysis},
    VOLUME = {11},
      YEAR = {2001},
    NUMBER = {3},
     PAGES = {589--608},
      ISSN = {1016-443X},
   MRCLASS = {58E15 (46N10 49Q20 53C20)},
  MRNUMBER = {1844080},
MRREVIEWER = {Lucio Renato Berrone},
       DOI = {10.1007/PL00001679},
       URL = {https://doi.org/10.1007/PL00001679},
}

@article {CMS01,
    AUTHOR = {Cordero-Erausquin, Dario and McCann, Robert J. and
              Schmuckenschl\"{a}ger, Michael},
     TITLE = {A {R}iemannian interpolation inequality \`a la {B}orell,
              {B}rascamp and {L}ieb},
   JOURNAL = {Invent. Math.},
  FJOURNAL = {Inventiones Mathematicae},
    VOLUME = {146},
      YEAR = {2001},
    NUMBER = {2},
     PAGES = {219--257},
      ISSN = {0020-9910},
   MRCLASS = {58E35 (28C99 60E15)},
  MRNUMBER = {1865396},
MRREVIEWER = {C\'{e}dric Villani},
       DOI = {10.1007/s002220100160},
       URL = {https://doi.org/10.1007/s002220100160},
}

@book {Vil09,
    AUTHOR = {Villani, C\'{e}dric},
     TITLE = {Optimal transport},
    SERIES = {Grundlehren der mathematischen Wissenschaften [Fundamental
              Principles of Mathematical Sciences]},
    VOLUME = {338},
      NOTE = {Old and new},
 PUBLISHER = {Springer-Verlag, Berlin},
      YEAR = {2009},
     PAGES = {xxii+973},
      ISBN = {978-3-540-71049-3},
   MRCLASS = {49-02 (28A75 37J50 49Q20 53C23 58E30)},
  MRNUMBER = {2459454},
MRREVIEWER = {Dario Cordero-Erausquin},
       DOI = {10.1007/978-3-540-71050-9},
       URL = {https://doi.org/10.1007/978-3-540-71050-9},
}

@article {KM18,
    AUTHOR = {Ketterer, Christian and Mondino, Andrea},
     TITLE = {Sectional and intermediate {R}icci curvature lower bounds via
              optimal transport},
   JOURNAL = {Adv. Math.},
  FJOURNAL = {Advances in Mathematics},
    VOLUME = {329},
      YEAR = {2018},
     PAGES = {781--818},
      ISSN = {0001-8708},
   MRCLASS = {53C21 (49Q20)},
  MRNUMBER = {3783428},
MRREVIEWER = {Martin Kell},
       DOI = {10.1016/j.aim.2018.01.024},
       URL = {https://doi.org/10.1016/j.aim.2018.01.024},
}

@article {Pas12,
    AUTHOR = {Pass, Brendan},
     TITLE = {Regularity of optimal transportation between spaces with
              different dimensions},
   JOURNAL = {Math. Res. Lett.},
  FJOURNAL = {Mathematical Research Letters},
    VOLUME = {19},
      YEAR = {2012},
    NUMBER = {2},
     PAGES = {291--307},
      ISSN = {1073-2780},
   MRCLASS = {49Q20 (35B65 49N60 90B06)},
  MRNUMBER = {2955762},
MRREVIEWER = {Flavia-Corina Mitroi-Symeonidis},
       DOI = {10.4310/MRL.2012.v19.n2.a3},
       URL = {https://doi.org/10.4310/MRL.2012.v19.n2.a3},
}

@article {MP20,
    AUTHOR = {McCann, Robert J. and Pass, Brendan},
     TITLE = {Optimal transportation between unequal dimensions},
   JOURNAL = {Arch. Ration. Mech. Anal.},
  FJOURNAL = {Archive for Rational Mechanics and Analysis},
    VOLUME = {238},
      YEAR = {2020},
    NUMBER = {3},
     PAGES = {1475--1520},
      ISSN = {0003-9527},
   MRCLASS = {49Q22 (49N60)},
  MRNUMBER = {4160804},
MRREVIEWER = {Luca Granieri},
       DOI = {10.1007/s00205-020-01569-5},
       URL = {https://doi.org/10.1007/s00205-020-01569-5},
}

@book {AGS05,
    AUTHOR = {Ambrosio, Luigi and Gigli, Nicola and Savar\'{e}, Giuseppe},
     TITLE = {Gradient flows in metric spaces and in the space of
              probability measures},
    SERIES = {Lectures in Mathematics ETH Z\"{u}rich},
 PUBLISHER = {Birkh\"{a}user Verlag, Basel},
      YEAR = {2005},
     PAGES = {viii+333},
      ISBN = {978-3-7643-2428-5},
   MRCLASS = {49-02 (28A33 35K55 35K90 49Q20 60B10)},
  MRNUMBER = {2129498},
}

@book {San15,
    AUTHOR = {Santambrogio, Filippo},
     TITLE = {Optimal transport for applied mathematicians},
    SERIES = {Progress in Nonlinear Differential Equations and their
              Applications},
    VOLUME = {87},
      NOTE = {Calculus of variations, PDEs, and modeling},
 PUBLISHER = {Birkh\"{a}user/Springer, Cham},
      YEAR = {2015},
     PAGES = {xxvii+353},
      ISBN = {978-3-319-20827-5},
   MRCLASS = {49-02 (35J96 49J45 49M29 58E50 90C05 90C48 91B02)},
  MRNUMBER = {3409718},
MRREVIEWER = {Luigi De Pascale},
       DOI = {10.1007/978-3-319-20828-2},
       URL = {https://doi.org/10.1007/978-3-319-20828-2},
}

@book {AFP00,
    AUTHOR = {Ambrosio, Luigi and Fusco, Nicola and Pallara, Diego},
     TITLE = {Functions of bounded variation and free discontinuity
              problems},
    SERIES = {Oxford Mathematical Monographs},
 PUBLISHER = {The Clarendon Press, Oxford University Press, New York},
      YEAR = {2000},
     PAGES = {xviii+434},
      ISBN = {0-19-850245-1},
   MRCLASS = {49-02 (49J45 49K10 49Qxx)},
  MRNUMBER = {1857292},
MRREVIEWER = {J. E. Brothers},
}

@book {Pet16,
    AUTHOR = {Petersen, Peter},
     TITLE = {Riemannian geometry},
    SERIES = {Graduate Texts in Mathematics},
    VOLUME = {171},
   EDITION = {Third},
 PUBLISHER = {Springer, Cham},
      YEAR = {2016},
     PAGES = {xviii+499},
      ISBN = {978-3-319-26652-7},
   MRCLASS = {53-01 (53C20 53C21 53C23)},
  MRNUMBER = {3469435},
       DOI = {10.1007/978-3-319-26654-1},
       URL = {https://doi.org/10.1007/978-3-319-26654-1},
}

@article {CM17,
    AUTHOR = {Cavalletti, Fabio and Mondino, Andrea},
     TITLE = {Sharp and rigid isoperimetric inequalities in metric-measure
              spaces with lower {R}icci curvature bounds},
   JOURNAL = {Invent. Math.},
  FJOURNAL = {Inventiones Mathematicae},
    VOLUME = {208},
      YEAR = {2017},
    NUMBER = {3},
     PAGES = {803--849},
      ISSN = {0020-9910},
   MRCLASS = {53C23 (49Q05 53C21)},
  MRNUMBER = {3648975},
MRREVIEWER = {Renjin Jiang},
       DOI = {10.1007/s00222-016-0700-6},
       URL = {https://doi.org/10.1007/s00222-016-0700-6},
}

@article {CNV04,
    AUTHOR = {Cordero-Erausquin, D. and Nazaret, B. and Villani, C.},
     TITLE = {A mass-transportation approach to sharp {S}obolev and
              {G}agliardo-{N}irenberg inequalities},
   JOURNAL = {Adv. Math.},
  FJOURNAL = {Advances in Mathematics},
    VOLUME = {182},
      YEAR = {2004},
    NUMBER = {2},
     PAGES = {307--332},
      ISSN = {0001-8708},
   MRCLASS = {26D15 (46E35)},
  MRNUMBER = {2032031},
MRREVIEWER = {Olivier Druet},
       DOI = {10.1016/S0001-8708(03)00080-X},
       URL = {https://doi.org/10.1016/S0001-8708(03)00080-X},
}

@article {Cas10,
    AUTHOR = {Castillon, Philippe},
     TITLE = {Submanifolds, isoperimetric inequalities and optimal
              transportation},
   JOURNAL = {J. Funct. Anal.},
  FJOURNAL = {Journal of Functional Analysis},
    VOLUME = {259},
      YEAR = {2010},
    NUMBER = {1},
     PAGES = {79--103},
      ISSN = {0022-1236},
   MRCLASS = {49Q20 (46E35 53C42)},
  MRNUMBER = {2610380},
MRREVIEWER = {Luca Granieri},
       DOI = {10.1016/j.jfa.2010.03.001},
       URL = {https://doi.org/10.1016/j.jfa.2010.03.001},
}

@article {CRS16,
    AUTHOR = {Cabr\'{e}, Xavier and Ros-Oton, Xavier and Serra, Joaquim},
     TITLE = {Sharp isoperimetric inequalities via the {ABP} method},
   JOURNAL = {J. Eur. Math. Soc. (JEMS)},
  FJOURNAL = {Journal of the European Mathematical Society (JEMS)},
    VOLUME = {18},
      YEAR = {2016},
    NUMBER = {12},
     PAGES = {2971--2998},
      ISSN = {1435-9855},
   MRCLASS = {49Q05 (28A75 35A23)},
  MRNUMBER = {3576542},
       DOI = {10.4171/JEMS/659},
       URL = {https://doi.org/10.4171/JEMS/659},
}

@article {WZ13,
    AUTHOR = {Wang, Yu and Zhang, Xiangwen},
     TITLE = {An {A}lexandroff-{B}akelman-{P}ucci estimate on {R}iemannian
              manifolds},
   JOURNAL = {Adv. Math.},
  FJOURNAL = {Advances in Mathematics},
    VOLUME = {232},
      YEAR = {2013},
     PAGES = {499--512},
      ISSN = {0001-8708},
   MRCLASS = {53C21 (35B45 35B65 35R01)},
  MRNUMBER = {2989991},
MRREVIEWER = {Xiang Gao},
       DOI = {10.1016/j.aim.2012.09.009},
       URL = {https://doi.org/10.1016/j.aim.2012.09.009},
}

@article {XZ17,
    AUTHOR = {Xia, Chao and Zhang, Xiangwen},
     TITLE = {A{BP} estimate and geometric inequalities},
   JOURNAL = {Comm. Anal. Geom.},
  FJOURNAL = {Communications in Analysis and Geometry},
    VOLUME = {25},
      YEAR = {2017},
    NUMBER = {3},
     PAGES = {685--708},
      ISSN = {1019-8385},
   MRCLASS = {53C21 (52A40 53C42)},
  MRNUMBER = {3702549},
MRREVIEWER = {B. S. Rubin},
       DOI = {10.4310/CAG.2017.v25.n3.a6},
       URL = {https://doi.org/10.4310/CAG.2017.v25.n3.a6},
}

@article {FMP10,
    AUTHOR = {Figalli, A. and Maggi, F. and Pratelli, A.},
     TITLE = {A mass transportation approach to quantitative isoperimetric
              inequalities},
   JOURNAL = {Invent. Math.},
  FJOURNAL = {Inventiones Mathematicae},
    VOLUME = {182},
      YEAR = {2010},
    NUMBER = {1},
     PAGES = {167--211},
      ISSN = {0020-9910},
   MRCLASS = {49Q20 (35J96)},
  MRNUMBER = {2672283},
MRREVIEWER = {Lorenzo Brasco},
       DOI = {10.1007/s00222-010-0261-z},
       URL = {https://doi.org/10.1007/s00222-010-0261-z},
}

@article {Tru94,
    AUTHOR = {Trudinger, Neil S.},
     TITLE = {Isoperimetric inequalities for quermassintegrals},
   JOURNAL = {Ann. Inst. H. Poincar\'{e} C Anal. Non Lin\'{e}aire},
  FJOURNAL = {Annales de l'Institut Henri Poincar\'{e} C. Analyse Non Lin\'{e}aire},
    VOLUME = {11},
      YEAR = {1994},
    NUMBER = {4},
     PAGES = {411--425},
      ISSN = {0294-1449},
   MRCLASS = {52A40 (35J60 52A39)},
  MRNUMBER = {1287239},
MRREVIEWER = {Jane R. Sangwine-Yager},
       DOI = {10.1016/S0294-1449(16)30181-0},
       URL = {https://doi.org/10.1016/S0294-1449(16)30181-0},
}

@article {Cab08,
    AUTHOR = {Cabr\'{e}, Xavier},
     TITLE = {Elliptic {PDE}'s in probability and geometry: symmetry and
              regularity of solutions},
   JOURNAL = {Discrete Contin. Dyn. Syst.},
  FJOURNAL = {Discrete and Continuous Dynamical Systems. Series A},
    VOLUME = {20},
      YEAR = {2008},
    NUMBER = {3},
     PAGES = {425--457},
      ISSN = {1078-0947},
   MRCLASS = {35J60 (35B65 58J05 60F05 60G40)},
  MRNUMBER = {2373200},
MRREVIEWER = {Nicolas Saintier},
       DOI = {10.3934/dcds.2008.20.425},
       URL = {https://doi.org/10.3934/dcds.2008.20.425},
}

@article {AFM20,
    AUTHOR = {Agostiniani, Virginia and Fogagnolo, Mattia and Mazzieri,
              Lorenzo},
     TITLE = {Sharp geometric inequalities for closed hypersurfaces in
              manifolds with nonnegative {R}icci curvature},
   JOURNAL = {Invent. Math.},
  FJOURNAL = {Inventiones Mathematicae},
    VOLUME = {222},
      YEAR = {2020},
    NUMBER = {3},
     PAGES = {1033--1101},
      ISSN = {0020-9910},
   MRCLASS = {53C40 (53C21 53C42)},
  MRNUMBER = {4169055},
MRREVIEWER = {Yun Tao Zhang},
       DOI = {10.1007/s00222-020-00985-4},
       URL = {https://doi.org/10.1007/s00222-020-00985-4},
}

@article {GM00,
    AUTHOR = {Gangbo, Wilfrid and McCann, Robert J.},
     TITLE = {Shape recognition via {W}asserstein distance},
   JOURNAL = {Quart. Appl. Math.},
  FJOURNAL = {Quarterly of Applied Mathematics},
    VOLUME = {58},
      YEAR = {2000},
    NUMBER = {4},
     PAGES = {705--737},
      ISSN = {0033-569X},
   MRCLASS = {94A08 (28A35 49Q20)},
  MRNUMBER = {1788425},
       DOI = {10.1090/qam/1788425},
       URL = {https://doi.org/10.1090/qam/1788425},
}

@misc{BK22,
  doi = {10.48550/ARXIV.2012.11862},
  
  url = {https://arxiv.org/abs/2012.11862},
  
  author = {Balogh, Zolt\'an M. and Krist\'aly, Alexandru},
  
  title = {Sharp isoperimetric and Sobolev inequalities in spaces with nonnegative Ricci curvature},
  
  publisher = {arXiv},
  
  year = {2022},
  
  copyright = {arXiv.org perpetual, non-exclusive license},
  note={To appear in Math. Ann.}
}

@article {Bre91,
    AUTHOR = {Brenier, Yann},
     TITLE = {Polar factorization and monotone rearrangement of
              vector-valued functions},
   JOURNAL = {Comm. Pure Appl. Math.},
  FJOURNAL = {Communications on Pure and Applied Mathematics},
    VOLUME = {44},
      YEAR = {1991},
    NUMBER = {4},
     PAGES = {375--417},
      ISSN = {0010-3640},
   MRCLASS = {46E40 (35Q99 46E99 49Q99)},
  MRNUMBER = {1100809},
MRREVIEWER = {Robert McOwen},
       DOI = {10.1002/cpa.3160440402},
       URL = {https://doi.org/10.1002/cpa.3160440402},
}

@article {Cha20,
    AUTHOR = {Chahine, Yousef K.},
     TITLE = {Volume estimates for tubes around submanifolds using integral
              curvature bounds},
   JOURNAL = {J. Geom. Anal.},
  FJOURNAL = {Journal of Geometric Analysis},
    VOLUME = {30},
      YEAR = {2020},
    NUMBER = {4},
     PAGES = {4071--4091},
      ISSN = {1050-6926},
   MRCLASS = {53C20},
  MRNUMBER = {4167274},
MRREVIEWER = {Stefano Pigola},
       DOI = {10.1007/s12220-019-00230-2},
       URL = {https://doi.org/10.1007/s12220-019-00230-2},
}

@misc{CZ22,
  doi = {10.48550/ARXIV.2205.12582},
  
  url = {https://arxiv.org/abs/2205.12582},
  
  author = {Cui, Jingshi and Zhao, Peibiao},
  
  title = {Locally constrained flows and sharp Michael-Simon inequalities in hyperbolic space},
  
  publisher = {arXiv},
  
  year = {2022},
  
  copyright = {arXiv.org perpetual, non-exclusive license},
  note={preprint},
}

@article {Rus96,
    AUTHOR = {R\"{u}schendorf, Ludger},
     TITLE = {On {$c$}-optimal random variables},
   JOURNAL = {Statist. Probab. Lett.},
  FJOURNAL = {Statistics \& Probability Letters},
    VOLUME = {27},
      YEAR = {1996},
    NUMBER = {3},
     PAGES = {267--270},
      ISSN = {0167-7152},
   MRCLASS = {62H20 (60E05)},
  MRNUMBER = {1395577},
MRREVIEWER = {Ramesh C. Gupta},
       DOI = {10.1016/0167-7152(95)00078-X},
       URL = {https://doi.org/10.1016/0167-7152(95)00078-X},
}

@article {FM22,
    AUTHOR = {Fogagnolo, Mattia and Mazzieri, Lorenzo},
     TITLE = {Minimising hulls, {$p$}-capacity and isoperimetric inequality
              on complete {R}iemannian manifolds},
   JOURNAL = {J. Funct. Anal.},
  FJOURNAL = {Journal of Functional Analysis},
    VOLUME = {283},
      YEAR = {2022},
    NUMBER = {9},
     PAGES = {Paper No. 109638, 49},
      ISSN = {0022-1236},
   MRCLASS = {49Q10 (31C15 49J40 53C21 53E10)},
  MRNUMBER = {4459004},
       DOI = {10.1016/j.jfa.2022.109638},
       URL = {https://doi.org/10.1016/j.jfa.2022.109638},
}

@misc{CM22,
  doi = {10.48550/ARXIV.2207.03423},
  
  url = {https://arxiv.org/abs/2207.03423},
  
  author = {Cavalletti, Fabio and Manini, Davide},
  
  title = {Rigidities of Isoperimetric inequality under nonnegative Ricci curvature},
  
  publisher = {arXiv},
  
  year = {2022},
  
  copyright = {arXiv.org perpetual, non-exclusive license}
}

@misc{APPS22,
  doi = {10.48550/ARXIV.2201.04916},
  
  url = {https://arxiv.org/abs/2201.04916},
  
  author = {Antonelli, Gioacchino and Pasqualetto, Enrico and Pozzetta, Marco and Semola, Daniele},
  
  title = {Sharp isoperimetric comparison on non collapsed spaces with lower Ricci bounds},
  
  publisher = {arXiv},
  
  year = {2022},
  
  copyright = {Creative Commons Attribution 4.0 International}
}

@article {GM96,
    AUTHOR = {Gangbo, Wilfrid and McCann, Robert J.},
     TITLE = {The geometry of optimal transportation},
   JOURNAL = {Acta Math.},
  FJOURNAL = {Acta Mathematica},
    VOLUME = {177},
      YEAR = {1996},
    NUMBER = {2},
     PAGES = {113--161},
      ISSN = {0001-5962},
   MRCLASS = {49Q20 (28A99 35Q99 60B99 65K10)},
  MRNUMBER = {1440931},
MRREVIEWER = {Ludger R\"{u}schendorf},
       DOI = {10.1007/BF02392620},
       URL = {https://doi.org/10.1007/BF02392620},
}

@article {Gig11,
    AUTHOR = {Gigli, Nicola},
     TITLE = {On the inverse implication of {B}renier-{M}c{C}ann theorems
              and the structure of {$(\mathscr P_2(M),W_2)$}},
   JOURNAL = {Methods Appl. Anal.},
  FJOURNAL = {Methods and Applications of Analysis},
    VOLUME = {18},
      YEAR = {2011},
    NUMBER = {2},
     PAGES = {127--158},
      ISSN = {1073-2772},
   MRCLASS = {49Q20 (28A33)},
  MRNUMBER = {2847481},
MRREVIEWER = {Luca Granieri},
       DOI = {10.4310/MAA.2011.v18.n2.a1},
       URL = {https://doi.org/10.4310/MAA.2011.v18.n2.a1},
}

@book {Coh13,
    AUTHOR = {Cohn, Donald L.},
     TITLE = {Measure theory},
    SERIES = {Birkh\"{a}user Advanced Texts: Basler Lehrb\"{u}cher. [Birkh\"{a}user
              Advanced Texts: Basel Textbooks]},
   EDITION = {Second},
 PUBLISHER = {Birkh\"{a}user/Springer, New York},
      YEAR = {2013},
     PAGES = {xxi+457},
      ISBN = {978-1-4614-6955-1},
   MRCLASS = {28-01},
  MRNUMBER = {3098996},
MRREVIEWER = {Ville Suomala},
       DOI = {10.1007/978-1-4614-6956-8},
       URL = {https://doi.org/10.1007/978-1-4614-6956-8},
}

\end{document}